%% file: campana-toric-arXiv.tex
\documentclass[12pt,reqno,a4paper]{amsart}
\usepackage[utf8]{inputenc}
\usepackage[headings]{fullpage} 
 
\usepackage[english]{babel}
\usepackage[T1]{fontenc}         
\usepackage{lmodern}             
\usepackage{amsmath,amsthm,amsfonts,amssymb,bm}           

\usepackage[backgroundcolor=white,linecolor=orange,bordercolor=orange,textsize=tiny]{todonotes}

\usepackage{mathrsfs}
\usepackage{graphicx}
\usepackage{tikz-cd}
\usepackage{rotating}

\usepackage{url}
\definecolor{blue-grey}{HTML}{4A90E2}

\usepackage[colorlinks,backref=page,allcolors=blue-grey]{hyperref}\usepackage[alphabetic,backrefs,lite]{amsrefs}
\usepackage[capitalize,nameinlink,noabbrev]{cleveref} 
\crefname{page}{page}{pages}

\newtheorem{mythm}{Theorem}[section]
\newtheorem{mycor}[mythm]{Corollary}
\newtheorem{mylemma}[mythm]{Lemma}

\newtheorem{mythmintro}{Theorem}
\newtheorem{myquest}{Question}

\newtheorem{myptn}[mythm]{Proposition}

\theoremstyle{definition}
\newtheorem{mydef}[mythm]{Definition}
\newtheorem{mydefintro}{Definition}

\newtheorem{myexample}[mythm]{Example}

\newtheorem{myremark}[mythm]{Remark}

\input{macros.tex}

\begin{document}

\title[Motivic countings of curves via universal torsors]{Motivic counting of rational curves \\ with tangency conditions \\ via universal torsors \\}

%%%%%%%%%%%%%%%%%%%%%%%%%%%%%%%%%%%%%%%%%%%%%%%%%%%%%%
%%%%%%%%%%%%%%%%%%%%% ABSTRACT %%%%%%%%%%%%%%%%%%%%%%%
%%%%%%%%%%%%%%%%%%%%%%%%%%%%%%%%%%%%%%%%%%%%%%%%%%%%%%

\begin{abstract}
 Using the formalism of Cox rings and universal torsors, 
 we prove a decomposition of the 
 Grothendieck motive
 of the 
 moduli space of morphisms 
 from an arbitrary smooth projective curve
 to a Mori Dream Space (MDS).

For the simplest cases of MDS, that of toric varieties, 
we use this decomposition
to prove 
an instance of the motivic Batyrev--Manin--Peyre principle for curves satisfying
tangency conditions with respect to the boundary divisors,
often called
Campana curves. 
\end{abstract}

%%%%%%%%%%%%%%%%%%%%%%%%%%%%%%%%%%%%%%%%%%%%%%%%%%%%%%
%%%%%%%%%%%%%%%%%%%%%% AUTHOR %%%%%%%%%%%%%%%%%%%%%%%%
%%%%%%%%%%%%%%%%%%%%%%%%%%%%%%%%%%%%%%%%%%%%%%%%%%%%%%

\author[L. Faisant]{Loïs Faisant}
\address{IST Austria, Am Campus 1, 3400 Klosterneuburg, Austria}
\email{lois.faisant@ista.ac.at or @m4x.org}
\date{\today}
%\date{\textbf{Preliminary version of \today. Please do not distribute without permission :)}}
\thanks{The author acknowledges funding from the European Union’s Horizon 2020 research and innovation programme under the Marie Skłodowska-Curie grant agreement No 101034413.}

\makeatletter
\@namedef{subjclassname@2020}{%
  \textup{2020} Mathematics Subject Classification}
\makeatother

\subjclass[2020]{14H10, 14E18, 11G50, 14G40, 14J45}
% 14E18   	Arcs and motivic integration
% 14H10 	Families, moduli of curves (algebraic)
% 11G50   	Heights 		[See also 14G40, 37P30]
% 14G40 	Arithmetic varieties and schemes; Arakelov theory; heights 				[See also 11G50, 37P30]
% 14J45   	Fano varieties 

% 14E30   	Minimal model program (related) 

\keywords{}

%%%%%%%%%%%%%%%%%%%%%%%%%%%%%%%%%%%%%%%%%%%%%%%%%%%%%%
%%%%%%%%%%%%%%%%%%%%%% TITLE %%%%%%%%%%%%%%%%%%%%%%%%%
%%%%%%%%%%%%%%%%%%%%%%%%%%%%%%%%%%%%%%%%%%%%%%%%%%%%%%

\maketitle

%%%%%%%%%%%%%%%%%%%%%%%%%%%%%%%%%%%%%%%%%%%%%%%%
%%%%%%%%%%%%%%%%%%%% TOC %%%%%%%%%%%%%%%%%%%%%%%
%%%%%%%%%%%%%%%%%%%%%%%%%%%%%%%%%%%%%%%%%%%%%%%%

\setcounter{tocdepth}{1}
\tableofcontents

%%%%%%%%%%%%%%%%%%%%%%%%%%%%%%%%%%%%%%%%%%%%%%%%
%%%%%%%%%%%%%%%%%%%% INTRO %%%%%%%%%%%%%%%%%%%%%
%%%%%%%%%%%%%%%%%%%%%%%%%%%%%%%%%%%%%%%%%%%%%%%%

\section*{Introduction}

% TODO [X] Rectifier : ce n'est pas le premier article du genre 
In this article, we aim
to illustrate how one can use the formalism of Cox rings and universal torsors 
to study the moduli space 
\[
\HHom_\kb ( \CCC , X )
\]
parametrising morphisms from an arbitrary smooth projective curve $\CCC$ 
to a Mori Dream Space $X$,
both varieties $X$ and $\CCC$ 
being defined above an
 arbitrary base field
$\kb$. 
More precisely, we would like to study its Grothendieck motive
and the approach we adopt is
guided by a so-called
motivic Batyrev-Manin-Peyre principle 
\cite[Question 2]{faisant2023motivic-distribution}.

\medskip

The present paper contains two main results,
\cref{thm-intro-parametrisation} \cpageref{thm-intro-parametrisation} 
and 
\cref{thm-intro-campana} \cpageref{thm-intro-campana}.
The first one
is an explicit algebraic parametrisation of morphisms $\CCC \to X$
with image 
intersecting 
the complement $U\subset X$
of the zero locus of a given set of generators of the Cox ring of $X$.
The second main result is an application
of this parametrisation
to the case where $X$ 
is a smooth projective split toric variety over $\kb$ 
and 
one restricts
to morphisms $\mathscr C \to X$
having their image
satisfying certain tangency conditions 
with respect to the boundary divisors of the toric variety.
Such morphisms are often called \emph{Campana curves}
and we prove 
a variant 
of the motivic Batyrev-Manin-Peyre principle 
for such morphisms. 

\medskip

In this introduction we quickly present our two results,
providing some context and illustrating 
how they are related to the domain's state-of-the art. 
The organisation of the paper is given on \cpageref{introduction-orga-paper}.

\subsection*{Mori Dream Spaces}
Let $X$ be a projective smooth and geometrically integral variety above a field $\kb$.
In this article we assume moreover that 
$\Pic ( X ) $ is a split free $\ZZ$-module,
that is to say that the absolute Galois action on it is trivial.
The so-called \emph{Mori Dream Spaces} \cite{HK2000mori}
are exactly the $X$'s 
whose Cox Ring 
\[
\bigoplus_{[L] \in \Pic ( X )} H^0 ( X , L ) 
\]
is finitely generated. Note that Fano varieties are known to be Mori Dream Spaces by \cite{BCHM10}. 
With the above definition of the Cox ring, the ring structure depends on the choice of a basis of $\Pic ( X ) $. 
We assume this choice to be made once and for all. 

\medskip

So, in this introduction, 
we assume that $X$ is 
a Mori Dream Space.
Its effective cone is rational polyhedral 
and we 
fix once and for all 
a finite family $( \mathcal D_i )_{i\in \mathfrak I}$
of effective divisors on $X$ whose classes generate the effective cone of $X$. 
%\begin{center}
%\begin{minipage}{\linewidth}% to keep image and caption on one page
%\makebox[\linewidth]{%        to center the image
%  \includegraphics{Drawing-Divisors.pdf}
%  }
%\end{minipage}
%\end{center}
Moreover, let $U$ be the open subset of $X$ 
obtained by deleting these effective divisors. 
There is a scheme 
\[
\HHom_\kb ( \CCC , X )_U
\]
parametrising morphisms $\CCC \to X$ 
whose image intersects the open subset $U$
and that is the one we are going to work with.

\subsection*{Motivic Batyrev-Manin-Peyre principle}
Our work 
is motivated by the study of the precise motivic distribution of curves on Fano varieties,
or varieties that are not far from being Fano.
As in our previous works \cite{faisant2023geometric-BMP-Ga,faisant2023motivic-distribution},
our approach is very much inspired by Manin's program \cite{franke1989rational,peyre1995hauteurs}
and the classical dictionary between number fields and function fields, which allows us to see $\ZZ$-points of arithmetic schemes as morphisms from a ``curve'' and morphisms from a $\kb$-curve to a $\kb$-variety as a ``point''.

\medskip 

We work within the Grothendieck ring
of varieties over $\kb$ (and later over an arbitrary scheme).
Such a ring $\KVar{\kb}$
is defined via generators and relations. Generators are isomorphism classes $[Y]$ of $\kb$-varieties 
and relations are 
\[
[ Y ] - [ Z ] - [ Y - Z ]
\] 
whenever $Z$ is a Zariski-closed subset of a $\kb$-variety $Y$. 
The multiplication law is given by taking Cartesian products of $\kb$-varieties. 
In this ring,
the class $\LL_\kb$ of the affine line $\mathbf A^1_\kb$
plays a particular role
and 
the ring of motivic integration $\MMM_\kb$
is obtained from $\KVar{\kb}$ by localising at $\LL_\kb$. 
An element of $\KVar{\kb}$ is called a \emph{class} or sometimes a \emph{Grothendieck motive}. 

\medskip 

Now, if we go back to our specific setting, 
it turns out that
for every $\dd \in \NN^\mathfrak I$,
morphisms $f : \CCC \to X$
having image 
intersecting $U$ and 
such that 
\[
 ( \deg ( f^* \mathcal D_i ) )_{i\in \mathfrak I} = \dd 
 \]
are parametrised by a $\kb$-scheme of finite type 
\[
\HHom^\dd_\kb  ( \CCC , X ) _U
\]
of dimension bounded below by 
\[
| \dd | + ( 1 - g ( \CCC ) ) \dim ( X )
\]
where $| \dd | = \sum_{i\in \mathfrak I} d_i$.
In particular,
it makes sense to consider its class in $\KVar{\kb}$ and $\MMM_\kb$.
Then, the motivic Batyrev-Manin-Peyre principle 
from \cite[Question 2]{faisant2023motivic-distribution}
predicts the behaviour of the normalised class
\[
\frac{
\left [
\HHom^\dd_\kb  ( \CCC , X ) _U 
\right ]}
{
\LL_\kb^{| \dd | + ( 1 - g ( \CCC ) ) \dim ( X ) }
} 
\in \MMM_\kb
\]
as ``$\dd \to \infty$'', meaning that 
all the coordinates of $\dd \in \NN^\mathfrak I$ tend to infinity.
It is expected to converge to a product of the form
\[
\left ( \frac{ \left [\PPic^0_{\CCC / \kb}\right ] \LL_\kb^{1-g ( \CCC ) } }{ \LL_\kb - 1 } \right )^{\rk ( \Pic ( V ))} 
	\prod_{p\in \CCC} 
	\left ( 1 - \LL_p^{-1} \right )^{\rk ( \Pic ( V ))} 
	\frac{[X]}{\LL_{p}^{\dim ( X )}} 
\]
where the second factor is a motivic Euler product in the sense of \cite{bilu2023MAMS}. 
Of course, we recall later in this paper the meaning of this notation, 
which is, 
as its name suggests, a motivic analogue of the classical notion of Euler products number theorists are familiar with. 

\medskip 

Therefore, if one wishes to prove that this expectation 
holds for a certain family of varieties, 
it is quite natural to start with trying to find an explicit description of the scheme $\HHom^\dd_\kb  ( \CCC , X ) _U$.
For example,
when $X$ is the projective plane, 
the moduli space of morphism $\PP^1_\kb \to \PP^2_\kb$
having image 
intersecting the torus $\{ x_0 x_1 x_2 \neq 0\}$  
is given, up to a multiplicative constant, by the space of tuples of three homogeneous polynomials, all of them being non-zero, of the same degree and mutually coprime. 
The parametrisation used in this paper broadly generalises this description. 

\subsection*{Lifting to the universal torsor}

Whenever a morphism $\CCC \to X$ is given, trying to lift it
to a space above $X$ whose geometry is expected to be simpler to understand seems to be a natural thing to do. 
A good candidate for such a nicer space is the universal torsor $\mathcal T_X$,
which is a torsor under the torus associated to $\Pic ( X )$.
\[
\begin{tikzcd}
& \mathcal T_X \dar["{/T_{\Pic ( X )}}"] \\
\CCC \rar \urar[dashed,"?"] & X 
\end{tikzcd}
\]
For each morphism $\CCC \to X$, it is possible to pull-back $\mathcal T_X$, and in case this morphism admits a lift, this lift induces a unique section of the pull-back.
\[
\begin{tikzcd}
f^* \mathcal T_X \rar \dar \arrow[dr, phantom, "\ulcorner", very near start] & \mathcal T_X \dar["{/T_{\Pic ( X )}}"] \\
\CCC \rar["f"]\arrow[u,shift right=1.5ex,dashrightarrow,"\exists !"']  \urar & X 
\end{tikzcd}
\]
In other words, our goal could be rephrased as follows: we would like to be able to parametrise every such lift
of every such torsor above $\CCC$. To do so, there exist some natural objects.

\subsection*{Ingredients of our parametrisation}

Our explicit parametrisation of morphisms 
\[{\CCC \to X}\]
from an arbitrary smooth projective $\kb$-curve $\CCC$
to a Mori Dream Space $X$
involves the following objects:
\begin{itemize}
	\item a presentation of the Cox ring of $X$ in terms of generators and relations, that is to say an isomorphism 
	\[
	\Cox ( X ) \simeq \kb [ ( s_i )_{i\in \mathfrak I } ] / \mathcal I_X
	\]
	where the $s_i$, $i\in \mathfrak I$, are sections defining the $( \mathcal D_i )_{i\in \mathfrak I}$ that are fixed once and for all;
	\item the Picard scheme $\PPic_{\CCC / \kb }$ of the smooth projective and geometrically irreducible $\kb$-curve $\CCC$, which exists under our assumptions;
	\item a universal sheaf $\mathscr P$ on  $\CCC \times \PPic_{\CCC / \kb }$, sometimes also called \emph{Poincaré sheaf}, which is an invertible sheaf satisfying a certain universal property;
	\item the scheme $\DDiv_{\CCC /\kb}$ parametrising effective divisors on $\CCC$ together with the Abel-Jacobi morphism
	\[
	\DDiv_{\CCC /\kb} \longrightarrow \PPic_{\CCC  / \kb} 
	\]
	sending an effective divisors $D$ to the class of $\mathscr O_\CCC ( D )$,
	which can be realised, as a $\PPic_{\CCC  / \kb} $-scheme\footnote{Up to working with a flattening stratification of a finite number of components of $\PPic_{\CCC  / \kb} $.}, as $\mathbf P ( \mathscr Q )$ for a certain coherent sheaf $\mathscr Q$ associated to $\mathscr P$,
	which means in particular that $\DDiv_{\CCC /\kb}$ admits a bundle of lines $\mathcal O_{\DDiv_{\CCC /\kb}} ( -1 )$,
	still relatively to $\PPic_{\CCC  / \kb} $;
	\item the space $	\HHom_\kb ( \CCC , X )_U 
$
can be seen as a scheme over $ ( \DDiv_{\CCC /\kb})^\mathfrak I$
via the morphism sending a curve $f : \CCC \to X$
to the $\mathfrak I$-tuple of effective divisors on $\CCC$ obtained by pulling back the $\mathcal D_i$'s via $f$.
\end{itemize}

\smallskip 

\begin{mythmintro}[\cref{thm-parametrisation-finale} \cpageref{thm-parametrisation-finale}]
\label{thm-intro-parametrisation}
	As a 
	scheme over the scheme $(\DDiv_{\CCC /\kb})^\mathfrak I$ parametrising  $\mathfrak I$-tuples of effective divisors on $\CCC$,
	the parameter space 
	\[
	\HHom_\kb ( \CCC , X )_U 
	\]
	is isomorphic to  
	an explicit locally closed subset 
	in  
	\[
	\underbrace{\mathcal O_{\DDiv_{\CCC /\kb}} ( -1 )^\times 
	\times_{\DDiv_{\CCC /\kb}} ... \times_{\DDiv_{\CCC /\kb}} \mathcal O_{\DDiv_{\CCC /\kb}} ( -1 )^\times}_{\mathfrak I \text{ times}}
	\]
	given by the ideal $\mathcal I_X$ 
	together with the coprimality conditions
	coming from the mutual non-intersections of the $\mathcal D_i$ for $i\in \mathfrak I$. 
\end{mythmintro}

When 
$X$ is a smooth projective split toric variety over $\kb$, 
we can directly use 
\cref{thm-intro-parametrisation} to 
prove the motivic Batyrev-Manin-Peyre principle for $\CCC \to X$,
generalising \cite[Theorem 5.4]{faisant2023motivic-distribution} (where $\mathscr C = \PP^1_\kb$)
to any $\CCC$,
see  
\cref{thm-BMP-C-toric} \cpageref{thm-BMP-C-toric}.
The question of equidistribution of curves
\cite[Theorem 5.6]{faisant2023motivic-distribution}
could also be addressed for a general $\CCC$ using \cref{thm-intro-parametrisation}, since most of the arguments from the proof in \cite[\S 5]{faisant2023motivic-distribution}
does not depend on the assumption $\CCC = \PP^1_\kb$. 

\subsection*{Campana curves on toric varieties}
Simply via pull-back, we are able
to apply our approach to a submoduli space of $\HHom_\kb ( \CCC , X )$ parametrising 
morphisms fulfilling prescribed tangency conditions with respect to the boundary divisors:
we give ourselves 
a tuple $\mm = ( m_i )_{i\in \mathfrak I} \in \NN^\mathfrak I$
and restrict to morphisms $\CCC \to X$ 
intersecting the boundary divisor $\mathcal D_i$
with local multiplicity at least $m_i$,
for every $i\in \mathfrak I$.
\begin{figure}[h]
\makebox[\linewidth]{%        to center the image
  \includegraphics{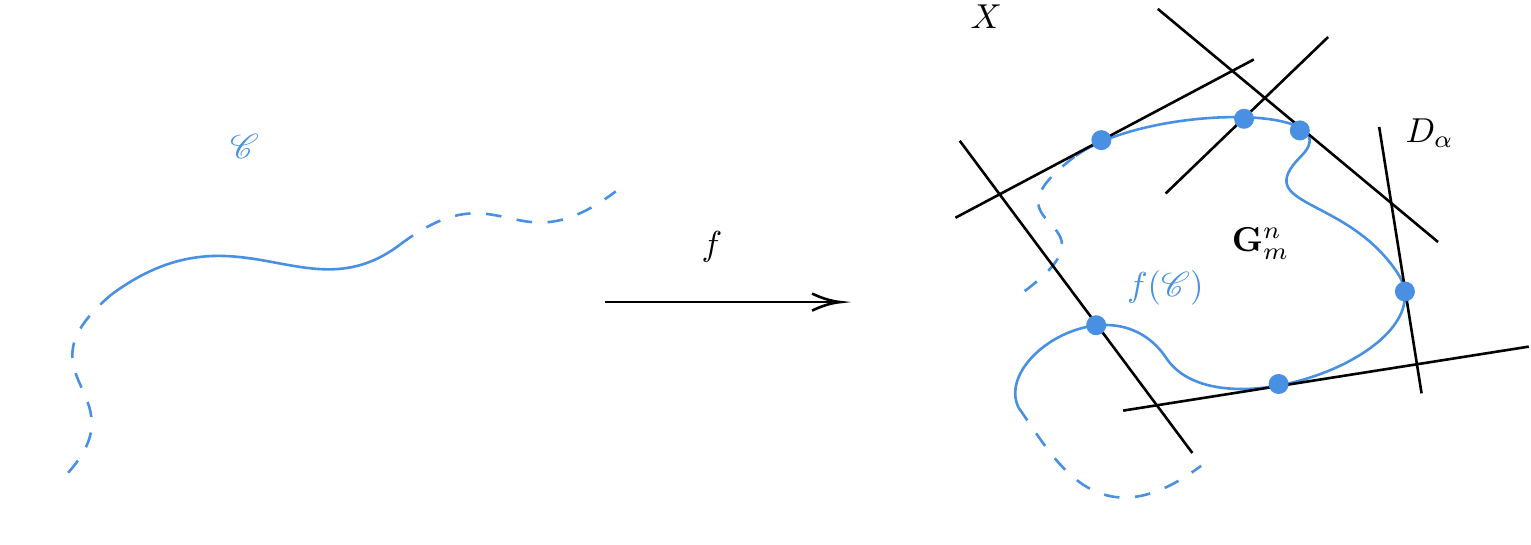}
  }
 \caption{An example of a Campana curve on a split projective toric variety. Here, the image of $\CCC$ by $f$ is tangent to four of the prime boundary divisors.}
 \label{figure-campana-toriques}
\end{figure}
Such morphisms are parametrised by a submoduli space 
\[
\HHom_\kb ( \CCC , ( X , \mathcal D_\epsilon ) ) 
\]
where $\mathcal D_{\bm \epsilon} $ is the $\QQ$-divisor
\[
\sum_{i\in \mathfrak I} \left ( 1 - \frac{1}{m_i} \right ) \mathcal D_i
\]
and $\bm \epsilon = ( 1 - 1 / m_i )_{i\in \mathfrak I}$.

Toric varieties are exactly the Mori Dream Spaces 
whose Cox ring has no relation at all \cite{cox1995homogeneous}.
Hence, they are the simplest example for which one can apply \cref{thm-intro-parametrisation}.
So let us 
assume that 
$X=X_\Sigma$ is a smooth split projective toric $\kb$-variety 
defined
by a complete regular fan $\Sigma$, 
the index set 
$\mathfrak I$ being in that case the set $\Sigma ( 1 ) $ of rays of $\Sigma$,
and $\mathcal D_\epsilon$ is the $\mathbf Q$-divisor
\[
\mathcal D_\epsilon = \sum_{\alpha \in \Sigma ( 1 )} \left ( 1 -  \frac{1}{m_\alpha} \right ) \mathcal D_\alpha 
\]
for a given tuple $(m_\alpha)_{\alpha \in \Sigma ( 1 )}$ of positive integers.

\smallskip

The second main result of this paper is the following 
instance of the motivic Batyrev-Manin-Peyre principle for Campana curves. 
It is stated in the dimensional completion of a finite extension
$\MMM_{\kb,r} = \MMM_\kb [T] / ( \LL_\kb - T^r ) $ of $\MMM_\kb$, where $r$ is a sufficiently large positive integer.

\begin{mythmintro}[\cref{thm-Campana-final} \cpageref{thm-Campana-final}]
	\label{thm-intro-campana}
	Let $X_\Sigma$ be a smooth split projective toric variety 
	over $\kb$, with split open torus $U \simeq \GG_m^n$
    and boundary divisors $\mathcal D_\alpha$, $\alpha \in \Sigma ( 1 )$.
    Let $m_\alpha$ be a finite family of positive integers indexed by the rays $\alpha \in \Sigma ( 1 )$
    and 
    \[
\mathcal D_\epsilon = \sum_{\alpha \in \Sigma ( 1 )} 
\left ( 1 -  \frac{1}{m_\alpha} \right )
\mathcal D_\alpha 
\]
with $\bm \epsilon = ( 1 - 1/m_\alpha )_{\alpha \in \Sigma ( 1 )}$.
	Then 
	\begin{align*}
	& \left [
	\HHom_\kb^\dd ( \CCC , ( X_\Sigma , \mathcal D_\epsilon ) )_U
	\right ]
	\LL_{\kb}^{- \dd \scdot ( \mathbf 1 - \bm{\epsilon} ) - ( 1 - g ( \CCC ) ) \dim ( X_\Sigma )}\\
	& \qquad \longrightarrow 
	\left ( \frac{ \left [\PPic^0_{\CCC / \kb}\right ] \LL_\kb^{1-g ( \CCC ) } }{ \LL_\kb - 1 } \right )^{\rk ( \Pic ( V ))} 
	\prod_{p\in \CCC} 
	( 1 - \LL_p^{-1} )^{\rk ( \Pic ( V ))} 
	\int_{
	\substack{
	\LLL_\infty ( X_\Sigma ) \\ \ord_{\mathcal D_\alpha} \in \ZZ_{m_\alpha} \cup \{ 0 \}
	}
	} 
	\mathrm d \mu_{( X_\Sigma , D_{\bm \epsilon} )}
	\end{align*}
	as  
	$\dd \to \infty $.
\end{mythmintro}

It is an analogue of \cite[Théorème F]{faisant2023thesis}. 
Concerning the general definition of the motivic integrals involved in the expression of the limit -- a motivic Tamagawa number -- we refer to
\cite{faisant2023motivic-distribution}
and 
\cite[Chap.1]{faisant2023thesis}. 
In this paper one only needs a ``toy version'' of these objects. 

\subsection*{Relation to other works}

There are at least three main techniques that are available to tacle Manin's conjecture:
harmonic analysis, the circle method or universal torsors.  
As its title suggests, this paper belongs to the category of works making use of the third. 

\subsubsection*{Manin--Peyre's conjecture over number fields via universal torsors}
So far, in this introduction we did not mention precisely how Cox rings 
are related to universal torsors,
as originally introduced by Colliot-Thélène--Sansuc in \cite{colliot-thelene-sansuc1987descente}.
We recall this link in \cref{section-cox-rings-MDS}.
The use of universal torsors in the setting of Manin's conjecture over number fields goes back to 
works of Peyre 
\cite{peyre1998asterisque-torseurs-univ}
and Salberger \cite{salberger1998tamagawa}. 
Since then, important contributions are due to Derenthal and Pieropan \cite{derenthal2009CountingIntegralPoints,pieropan2015torsors,derenthalpieropan2020splittorsor}. 
Very recently, 
Bongiorno \cite{bongiorno2024multi-height-toric-Q}
revisited the case of toric varieties over $\QQ$ 
via the new multi-height approach.

\subsubsection*{Campana points}
Over the last past few years, the literature about Campana points has been constantly growing. 
They are the arithmetic analogue of Campana curves,
when one works with $\Spec ( \mathbf Z )$ or the spectrum of the ring of integers of a number field in place of our curve $\CCC$.

In \cite{pieropan2021campana} formulated predictions, understood as a variant extending Manin's conjecture,
and checked them for equivariant compactifications of vector spaces,
building on previous work by Chambert-Loir \& Tschinkel \cite{chambert2012integral}.

Before that,
seminal works were due to
Van Valckenborgh concerning squarefull numbers in hyperplanes
\cite{vanvalckenborgh2012squareful}
and 
Browning \& Van Valckenborgh concerning squarefull numbers $x,y,z$ such that ${x+y=z}$ \cite{browning-vanvalckenborgh2012sums}.
In more recent years, Pieropan and Schindler 
checked the PSTVA conjecture for 
certain complete split toric varieties over $\QQ$ \cite{pieropan-schindler2024hyperbola} (using a generalised hyperbola method),
while Xiao treated the case of biequivariant compactifications of the Heisenberg group over $\QQ$ 
\cite{xiao2022campana}
(using the height zeta function),
Browning \& Yamagishi \cite{browning-yamagishi2021arithmetic} generalised the results of \cite{vanvalckenborgh2012squareful} (using again the circle method),
again extended to certain diagonal hypersurfaces of degree > 1 by
Balestrieri--Brandes--Kaesberg--Ortmann--Pieropan--Winter in  \cite{balestrieri2023campana}. 
Nevertheless, works of Shute
\cite{shute2022leading} 
and Streeter \cite{streeter2022campana}
provide some explicit examples
for which 
the constant appearing in the predictions of \cite{pieropan2021campana} 
is incorrect.
The recent preprint \cite{chow-loughran-takloo-tanimoto2024campana}
suggests a correction to the predictions from \cite{pieropan2021campana} 
and Shute--Streeter \cite{shute-streeter2024semi-integ-points-toric} checked it for toric varieties over number fields,
adapting the harmonic analysis method from \cite{batyrev-tschinkel1998manin-toric}. 
Some generalisations are considered by Moerman in \cite{moerman2024generalizedCampana}. 

\bigskip 

So far, to the author's knowledge, 
the study of Campana points over functions fields of positive characteristics, 
which correspond to sections of a model above a curve satisfying certain tangency conditions, has not been initiated yet. 
The present work
allows one to consider the case $\kb = \CC$ and $\kb = \mathbf F_q$ 
at the same time. For the second case, one can apply the counting 
measure 
\[
\begin{array}{rll}
	\KVar{\mathbf F_q} & \longrightarrow & \ZZ\\
	{[ X ]} 		   & \longmapsto 	 & \# X ( \mathbf F_q ) 
\end{array}
\]
to the proof of \cref{thm-intro-campana}
and deduce an $\mathbf F_q ( \CCC )$-version of our result.

Over $\mathbf C$,
motivic stabilisation of moduli spaces of 
Campana curves on equivariant compactifications of vector groups 
have been studied in \cite[Chap.~4]{faisant2023thesis}. 

\subsubsection*{Finite fields and Bourqui's previous works}
When $\kb = \mathbf F_q$ is a finite field, this decomposition can be seen as a lift to
$\KVar {\mathbf F_q}$
of earlier work of David Bourqui 
\cite{bourqui2009eclate3alignes}
who used the formalism of Cox rings 
to parametrise $\mathbf F_q$-points of $\HHom_{\mathbf F_q}^d ( \CCC , X )_U$, hence only at a set-theoretic level,
and applied it to rewrite 
the 
anticanonical height zeta function 
\[
\sum_{d\in \NN} \# \HHom_{\mathbf F_q}^d ( \CCC , X )_U ( \mathbf F_q )
\, 
\mathrm T^d 
\in \ZZ [[ \mathrm T ]] .
\]
Our algebraic parametrisation also generalises
the one given by Bourqui in \cite[Proposition 5.14]{bourqui2009produit}
when $X = X_\Sigma$ is a smooth split projective variety and $\CCC = \PP^1_\kb$,
which was inspired by Cox's description of the functor of points of a toric variety
\cite{cox1995functor}. It is this point of view as well that we adopt in this paper.

\subsubsection*{Motivic harmonic analysis}

In a parallel joint work with Margaret Bilu 
\cite{bilu-faisant2024poisson}
we develop motivic harmonic analysis techniques
to study the height zeta function, thanks to a new motivic Poisson formula, and apply it to study the motivic height zeta function of toric varieties. 

%%%%%%%%%%%%%%%%%%%%%%%%%%%%%%%%%%%%%%%%%%%%%%%%
%%%%%%%%%%%%%%%%%%%%% ORG %%%%%%%%%%%%%%%%%%%%%%
%%%%%%%%%%%%%%%%%%%%%%%%%%%%%%%%%%%%%%%%%%%%%%%% 

\subsection*{Organisation of the paper}
This paper is based on a variety of tools coming from classical algebraic geometry, number theory and motivic integration. 
We collect these tools in the first sections, while the proofs of our main results occupy the last sections. 
\label{introduction-orga-paper}

\medskip 

We recall generalities 
about rings of varieties, motivic functions 
and motivic Euler products 
in \cref{section-rings-varieties}. 

A reminder about flattening stratification can be found in \cref{section-geometric-realisation-coh-sheaves}
and we collect everything we need about Picard schemes, schemes of divisors and linear systems in a dedicated
\cref{section-picard}.

The motivic Batyrev-Manin-Peyre principle is recalled in 
\cref{section-motivic-BMP},
in a simplified version which is sufficient for our isotrivial setting.

What the reader needs to know about Cox rings and universal torsors is recalled in \cref{section-cox-rings-MDS}. 

The parametrisation
of morphisms $\CCC \to X$ to a Mori Dream Space is
given by \cref{thm-parametrisation-finale}
 in \cref{section-parametrisation-curves-MDS}.
It is used 
in \cref{section-toric-classic}
to prove a strong version of the motivic Batyrev-Manin-Peyre principle
for curves $\CCC \to X_\Sigma$ where the target $X_\Sigma$ is a smooth projective split toric variety,
\cref{thm-BMP-C-toric}.

We finally apply our parametrisation to study the motivic distribution of Campana curves
on toric varieties in \cref{section-Campana-curves-on-split-toric-varieties},
proving \cref{thm-intro-campana} as \cref{thm-Campana-final}.

%In \cref{section-intrinsic-hypersurfaces}
%we explain how our parametrisation can be used to prove 
%the motivic Batyrev-Manin-Peyre 
%principle 
%for Mori Dream Spaces 
%whose Cox ring 
%is a quotient of a polynomial ring by a principal ideal.

%%%%%%%%%%%%%%%%%%%%%%%%%%%%%%%%%%%%%%%%%%%%%%%%
%%%%%%%%%%%%%%%%%%%%% THX %%%%%%%%%%%%%%%%%%%%%%
%%%%%%%%%%%%%%%%%%%%%%%%%%%%%%%%%%%%%%%%%%%%%%%% 

\subsection*{Acknowledgements}

% MB, DB (not the Deutsch Bahn), TB, JG, EP, ...
The author warmly thanks Margaret Bilu for her interest in this work 
and the numerous discussions the last past years, 
David Bourqui for several discussions which helped the author with clarifying the state-of-the-art concerning applications of Cox rings to the functional Manin-Peyre conjecture, 
Tim Browning for his interest, support and comments on an earlier draft version of this work, 
Jakob Glas and Christian Bernert for their interest in this work, 
and Emmanuel Peyre for having introduced the author to 
this very broad and fascinating subject
already some years ago. 
Finally, thanks also go to Matilde Maccan,
Michel Raibaut and Tanguy Vernet for several feedbacks and comments on some parts of a previous draft of this work.

\subsection*{Convention} We use different fonts to differentiate line bundles, as $\mathcal L $ or the trivial one $\mathcal O_S$, 
from
invertible sheaves, as $\LLL , \PPP$ or $\OOO_S$.
When working above our curve $\CCC$, since it is assumed to be smooth, we can freely switch between line bundles and invertible sheaves on it.

\bigskip 

%%%%%%%%%%%%%%%%%%%%%%%%%%%%%%%%%%%%%%%%%%%%%%%%
%%%%%%%%%%%%%%%%%%% SECTION %%%%%%%%%%%%%%%%%%%%
%%%%%%%%%%%%%%%%%%%%%%%%%%%%%%%%%%%%%%%%%%%%%%%%

\section{Rings of varieties and motivic Euler products}

In this first section we recall the classical setup 
of motivic integration \cite{chambert-loir-nicaise-sebag2018motivic}
and the efficient machinery of motivic Euler products \cite{bilu2023MAMS}. 

\label{section-rings-varieties}

\subsection{Monoids and rings of varieties}
A general reference for this section is 
the second chapter of the monograph \cite{chambert-loir-nicaise-sebag2018motivic}.
Let $S$ be a scheme. 

\begin{mydef}
An \emph{$S$-variety} is a finitely presented morphism of schemes $X\to S$.
Morphisms between $S$-varieties are given by morphisms of {$S$-schemes}. All together, this defines the category
$\mathbf{Var}_S$
 of $S$-varieties. 
\end{mydef}

\begin{mydef}
    The \emph{Grothendieck monoid $\KVarmon S$ of varieties over $S$}
    (or monoid of $S$-varities, for short)
    is the free commutative 
    monoid
    generated by isomorphism classes $[ X \to S ]$ of $S$-varieties,
    modulo the relation
    \[
    [ \varnothing ] \sim 0
    \]
    and the so-called cut-and-past relations
    \[
    [ X ] \sim [ Y ] + [ X - Y ]
    \]
	whenever 
	$X$ is an $S$-variety and $Y$ is a finitely presented closed subscheme of $X$. 
	
	The \emph{Grothendieck group $\KVar S$ of varieties over $S$}
	is the commutative group associated to the monoid $\KVarmon S$. 
	It can be endowed with a structure of ring by setting 
	\[
	[ X ] \cdot [ Y ] 
	= [ X \times_S Y ] .
	\]
	The class of the affine line $\mathbf A^1_S$ in $\KVarmon S$ and $\KVar{S}$ will be denoted by $\mathbf L_S$.
    \end{mydef}

\begin{myremark}
	It is important for us to interpret classes above a scheme $S$
	as \emph{motivic functions}
	defined on this scheme. 
	Indeed, every point $s\in S$
	provides an evaluation map 
	\[
	\KVar S \longrightarrow \KVar{\kappa ( s )}
	\]
	sending the class of a variety $X \to S$ 
	to the class of its schematic fibre $X_s$ above the point $s$. 
	Moreover, a class $\varphi$ above $S$ (a motivic function on $S$) 
	is entirely determined by its values $\varphi ( s )$ at each point of $S$,
	as the following lemma ensures.
\end{myremark}

\begin{mylemma}
	Let $\varphi $ be a motivic function on $S$,
	that is to say for us an element of $\KVar{S}$.
	If $\varphi ( s ) = 0 $ 
	in $\KVar{\kappa ( s)}$
	for all point $s \in S$ then $\varphi = 0$ in $\KVar S$.  
\end{mylemma}

\begin{proof}
	See \cite[Lemma 1.1.8]{chambert-loir-loeser2016motivic}
	or 
	\cite[Theorem 1]{cluckers-halupczok2022evaluation}.
\end{proof}

\begin{mydef}
The ring $\mathscr M_S$ 
is the localisation 
\[
\mathscr M_S
= \KVar S [ \LL_S^{-1} ] .
\]
It is equipped with a decreasing filtration based on virtual dimension:  
for $m \in \ZZ$, let $\mathcal F^m \MMM_S$ denote the subgroup of $\MMM_S$ generated by elements of the form  
\[
[ X ]\LL_S^{-i},
\]  
where $X$ is an $S$-variety and $i$ is an integer satisfying $\dim_S (X) - i \leqslant -m$.  

The completion of $\MMM_S$ with respect to this decreasing filtration is expressed as the projective limit:  
\[
\widehat{\MMM_S}^{\dim } = 
\underset{\longleftarrow}{\lim } \, \MMM_S / \mathcal F^m \MMM_S
\]  
and there is a surjective morphism $\MMM_S \to \widehat{\MMM_S}^{\dim }$.  
\end{mydef}

In our study of the moduli space of Campana curves on toric varieties, 
we are naturally led to introduce rational powers of the affine line over a given scheme $S$. 
Indeed, the degrees under consideration may take non-integer rational values. 
Consequently, we may work within the quotient ring 
\[
\KVarpar{S,r} = \KVarpar{S}[x] / (x^r - \LL_S),
\] 
where $r$ is a nonzero natural number (in practice, chosen sufficiently large). 
Its localization at $\LL_S$ will be denoted by $\MMM_{S,r}$. 
Once again, we have a decreasing dimensional filtration $\FFF^m \MMM_{S,r}$, 
this time indexed by $m \in \frac{1}{r}\ZZ$, along with the associated completion: 
\[
\widehat{\MMM_{S,r}}^{\dim} = \underset{\longleftarrow}{\lim} \, \MMM_{S,r} / \FFF^m \MMM_{S,r}.
\]

When working in positive characteristic, for technical reasons 
\cite[Remark 2.0.2]{bilu-howe2021motivic-statistics}
we may replace $\KVar{\kb}$
and $\KVar{S}$ by 
a variant where \emph{universally homeomorphic} varieties are identified,
see \S 4.4 of \cite[Chap. 2]{chambert-loir-nicaise-sebag2018motivic}.
It is equivalent 
to taking the quotient with respect to \emph{radicial surjective morphisms}
\cite[Remark 2.1.4]{bilu-howe2021motivic-statistics}.

\subsection{Motivic euler products}
In this article we use Bilu's notion of motivic Euler products
introduced in \cite{bilu2023MAMS}. We very quickly recall here how 
this notion is defined
in terms of symmetric powers
and 
in the special case of families indexed by the non-zero elements of an additive monoid. 
Note again that one can interpret their coefficients as motivic functions. 
We mostly follow the exposition from \cite[\S 6.1]{bilu-howe2021motivic-statistics}. 

\medskip

If $n$ is a non-negative integer,
the $n$-th symmetric power of $X /S $
is the quotient 
\[
\Sym_{X/S}^n ( X ) 
= 
\underbrace{X \times_S ... \times_S X}_{n \text{ times}}
/ \mathfrak S_n 
\]
and if $Y \to X$ is another variety,
one gets the relative symmetric product 
\[
\Sym^n_{X/S} ( Y ) 
=
( 
\Sym^n_{/S} ( Y )
\longrightarrow  
\Sym^n_{/S} ( X ) ) .
\]
Now, one remarks that the 
product 
\[
\KVar{\Sym^\bullet_{/S} ( X ) }
=
\prod_{n \in \NN}
\KVar{\Sym^n_{/S} ( X )}
\]
can be endowed with a structure of graded $\KVar S$-algebra,
with a Cauchy product law
\[
\left 
( 
(ab)_n
\right )_{n\in \NN}
=
\left 
( 
\sum_{i=0}^n a_i b_{n-i}
\right )_{n\in \NN}
\]
where for every $n\in \NN$ and $i \in \{ 0 , ... , n \}$
the product comes from
the exterior product
\[
 a_i \boxtimes b_{n-i} \in \KVar{\Sym^i_{/S} ( X ) 
\times_\kb
\Sym^{n-i}_{/S} ( X ) }
\]
composed with the natural arrows 
\[
\Sym^i_{/S} ( X ) 
\times_\kb
\Sym^{n-i}_{/S} ( X ) 
\longrightarrow 
\Sym^n_{/S} ( X ) . 
\]
This allows one to extend the construction of symmetric products to classes in $\KVar{X}$.

\begin{mydef}
If ${\mu = ( m_i )_{i\in I} \in \NN^{(I)}}$, 
we set 
\[
\Sym_{S}^\mu ( X )
=
\prod_{i \in I} \Sym_{X/S}^{m_i} ( X ) 
\]
and if moreover $\mathscr A = ( a_i )_{i\in I}$
is a family of classes in $\KVar{X}$, we define
\[
\Sym_{X/S}^\mu ( \mathscr A )
= 
\boxtimes_{i\in I} \Sym_{X/S}^{m_i} ( a_i ) 
\]
in $\KVar{\Sym_{/S}^\mu ( X ) }$.

Elements of $\NN^{(I)}$ are often called \emph{generalised partitions}.

\end{mydef}

\begin{mydef}
	\label{def-mot-Eul-prod}
	Let $\mathscr X = ( X_i )_{i\in I}$
	be a family of motivic classes above $X$
	indexed by a set $I$.
	The formal motivic Euler product 
	\[
	\prod_{x\in X/S}
	\left ( 
	1 + \sum_{i\in  I} X_{i,x} \TT^i 
	\right )
	\]
	is by definition the formal series
	\[
	\sum_{\mu \in \NN^{(I)}} 
	\Sym_{X/S}^\mu ( \mathscr X )_* \, 
	\TT^\mu 
	\]	
	where each coefficient is seen above $\Sym^\mu_{/S} ( X ) $. 
\end{mydef}

\begin{myptn}[Multiplicativity]
	\label{ptn-Eul-prod-mult}
	Let $( X_m )_{m \in M}$
	and $( Y_m )_{m \in M}$
	be two families of $X$-varieties 
	indexed by an additive monoid $M$, with $X_0$ and $Y_0$ being equal to $1=[X/X]$. 
	Then 
		\begin{align*}
	& \prod_{x\in X/S}
	\left ( 
	\sum_{m\in M} X_{m,x} \TT^m
	\right )
	\times 
	\prod_{x\in X/S}
	\left ( 
	\sum_{m\in M} Y_{m,x} \TT^m  
	\right ) \\
	& =
	\prod_{x\in X/S}
	\left ( 
	\sum_{
	( m , m ' ) \in M^2
	} [ X_{m,x} \boxtimes_X Y_{m',x}] \TT^{m+m'}
	\right ).
	\end{align*} 
\end{myptn}

\subsection{Convergence of motivic series}

\label{section-lemmas-motivic-series}

\begin{mylemma}
	\label{lemma:convergence-criterion-Z-coeff-multivariable}
	Let $\uu = ( u_\alpha )_{\alpha \in \AAA}$
	be a family of indeterminates 
	indexed by a finite set $\AAA$
	and 
	$P( \uu ) \in \cup_{\mm \in \ZZ^\AAA_{>0}} \ZZ [[ \uu^{1/\mm} ]] $ 
	be a power series such that 
	\[
	P( \uu ) 
	\in 
	1 
	+ 
	\sum_{\alpha_1 \neq \alpha_2} 
	u_{\alpha_1} u_{\alpha_2}
	\cup_{\mm \in \ZZ^\AAA_{>0}} \ZZ [[ \uu^{1/\mm} ]] .
	\]
	Then the dimension of 
	the multidegree $\ee \in \QQ_{\geqslant 0}^\AAA $ coefficient of 
	the motivic Euler product 
	of 
	\[
	\prod_{x\in X } P ( \uu ) 
	\]
	is bounded by 
	\[
	| \ee | \dim ( X ) / 2 . 
	\]
\end{mylemma}

\begin{proof}
	We fix a numbering for elements of $\AAA$ 
	and denote by $e_1 , ... , e_{|\AAA|} \in \NN^\AAA$ the canonical basis of $\ZZ^{\AAA}$.
	The multidegree $\ee \in \QQ_{>0}^\AAA$ coefficient 
	is a sum of $\varpi$-th symmetric products over the partitions $\varpi = ( n_i )_{i \in \QQ_{>0}^\AAA}$ of $\ee$
	such that $n_i = 0$ whenever $ i \ngeqslant e_k + e_\ell$ for some $k\neq l$.
	In particular, $n_i \neq 0$ implies $| i | \geqslant 2$.
	The dimension of such a symmetric product is 
	\[
	\sum_{i\in \QQ_{>0}^\AAA} n_i \dim ( X ) 
	\leqslant \sum_{i\in \QQ_{>0} ^\AAA} \frac{| i |}{2} n_i \dim ( X )  = \frac{| \ee |}{2} \dim ( X )  
	\]
	hence the lemma. 
\end{proof}

\begin{mydef}
    We will say that a series 
     \[
F ( \TT )
= 
\sum_{\dd \in \NN^r}
\mathfrak a_\dd 
\TT^\dd 
    \]
converges absolutely for $ | t_i | < \LL^{ - \alpha_i  }$
if 
\[
\limsup_{\langle \alpha , \dd \rangle \neq 0}
\frac{
\dim_S ( \mathfrak a_\dd )}{\langle \alpha , \dd \rangle}
< 1 .
\]
\end{mydef}

\section{Geometric realisation of coherent sheaves}

\label{section-geometric-realisation-coh-sheaves}

This section is a recollection of useful recipes for coherent sheaves. 

\subsection{Flattening stratification}

The general definition of what is a flattening stratificiation
is given in \cite[\href{https://stacks.math.columbia.edu/tag/052F}{Section 052F}]{stacks-project} 
but we are going to use it in a very practical situation.
Another possible reference is 
\cite[Exposé IV]{SGA4}.

If $S$ is a scheme 
and $\mathscr F$
is a quasi-coherent $\mathscr O_S$-module of finite type,
the Fitting ideals of $\mathscr F$
\[
0 
= 
\mathrm{Fit}_{-1} ( \mathscr F )
\subset 
\mathrm{Fit}_{0} ( \mathscr F )
\subset 
\mathrm{Fit}_{1} ( \mathscr F )
\subset ...
\subset \mathscr O_S
\]
can be defined in terms of local presentations of $\mathscr F$. 
If 
\[
\oplus_{i\in \mathfrak I} \mathscr O_U 
\longrightarrow 
\mathscr O_U^{\oplus n}
\longrightarrow
\mathscr F_{| U}
\longrightarrow 
0
\]
is such a local presentation above an open subset $U\subset S$,
then the $r$-th Fitting ideal 
$\mathrm{Fit}_r ( \mathscr F )$
(where $r$ stands for ``rank'')
is generated by the $(n-r)\times(n-r)$-minors of the first arrow of the local presentation above. 

We are going to use the following lemma,
which is a mix of 
\cite[\href{https://stacks.math.columbia.edu/tag/05P9}{Lemma 05P9}]{stacks-project}
and
\cite[\href{https://stacks.math.columbia.edu/tag/05P8}{Lemma 05P8}]{stacks-project}. 

\begin{mylemma}
    \label{lemma-flattening-stratification-of-finitely-presented-sheaves}
    Let $S$ be a scheme and $\mathscr F$ be an $\mathscr O_S$-module of finite presentation.
    Let 
    \[
    S = Z_{-1}  
    \supset 
    Z_0
    \supset 
    Z_1
    \supset 
    ... 
    \]
    be the closed subschemes
    defined by the Fitting ideals of $\mathcal F$. 
    Let \[
    S_r = Z_{r-1} \setminus Z_r
    \]
    for all $r \in \NN$. 

    Then the morphisms $S_r \to S$ are of finite presentation
    and $\mathscr F_{|S_r}$ is locally free of rank $r$. 
\end{mylemma}

\subsection{Geometric realisation of coherent sheaves}

In \cite[\S 4]{bilu-howe2021motivic-statistics}
Bilu and Howe define 
a very convenient functor 
\[
\mathbf V : \mathbf{Coh}_S
\longrightarrow
\mathbf{Var}_{S-PW} 
\]
from the category $\mathbf{Coh}_S$ of coherent sheaves on $S$
to the category $\mathbf{Var}_{S-PW}$ of $S$-varieties localised at piecewise isomorphisms. 
At the level of locally
free sheaves,
this functor sends 
a locally free sheaf $\mathscr F$
to the geometric vector bundle 
\[
(\mathbf V ( \mathscr F ) \to S ) = (\Spec_S ( \Sym^\bullet ( \mathscr F^\vee ) ) \to S ) . 
\]
Using the canonical flattening stratification from \cref{lemma-flattening-stratification-of-finitely-presented-sheaves}
one can extend this operation to any coherent sheaf on $S$.
The image of a coherent sheaf is then piece-wisely isomorphic to 
\[
\mathbf A^{\dim ( \mathscr F )} \to S 
\] 
where $\dim ( \mathscr F )$ is the function $s \mapsto \dim_{\kappa ( s )} \mathscr F_s$.

\section{Picard schemes, schemes of linear systems and effective divisors}

\label{section-picard}

Following Kleiman \cite{kleiman2005picard}, we fix a locally Noetherian scheme $S$ and work in the category of locally Noetherian $S$-schemes. 
For most of our applications we will take ${S = \Spec ( \kb )}$ but working in a relative setting allows for more flexibility at some point,
for example to prove \cref{proposition-divisors-as-proj} below. 
Another reference is \cite[Chap. 8]{bosch-lutkebohmert-raynaud1990neron}.

\subsection{Some functors and representability criterions}
We recall some classical definitions \cite{kleiman2005picard}.

\begin{mydef}Let $f : X \to S $ be separated map of finite type.
The absolute Picard functor is the commutative group functor
\[
\Pic_{X} : 
( T \to S ) \longmapsto \Pic ( X \times_S T )
\]
and the relative Picard functor is the commutative group functor
\[
\Pic_{X / S} : 
( T \to S ) \longmapsto \Pic ( X \times_S T ) / \Pic ( T )  .
\]	
We thus have a morphism of functors 
\[
\Pic_X \longrightarrow \Pic_{X/ S} . 
\]
\end{mydef}

It is sometimes more convenient to work with a rigidified version 
of the previous functor.
\begin{mydef}
Assume that $f_* ( \mathscr O_X ) = \mathscr O_S$
holds universally and that $f$ admits a section $\varepsilon : S \to X$. 
For any invertible sheaf $\mathscr L $ on $X$,
a \emph{rigidification} is an isomorphism 
\[
\alpha : 	\mathscr O_S \overset{\sim}{\longrightarrow} \varepsilon^* \mathscr L . 
\] 
An isomorphism between rigidified invertible sheaves
is an isomorphism of invertible sheaves compatible with the rigidifications. 

The group functor $( P , \varepsilon )$ is the functor sending an $S$-scheme $T$ 
to the set of isomorphism classes of invertible sheaves on $X \times_S T$ rigidified along $\varepsilon_T : T \to X \times_S T$. 
\end{mydef}

\begin{myremark}
	Under the previous hypothesis, the functor $( P , \varepsilon )$ is automatically a sheaf for the Zariski topology, because rigidified line bundles do not admit non-trivial automorphisms,
	and 
	%the sequence 
	%\[
	%\begin{tikzcd}
	%	0 \rar &
	%	\Pic ( T ) 
	%	\rar 
	%	& \Pic ( T \times_S T ) 
	%	\rar &
	%	\Pic_{X/S} ( T ) \rar & 
	%	0
	%\end{tikzcd}
	%\] 
	%is exact and 
	$\Pic_{X/S} ( T ) = ( P , \varepsilon ) ( T )$ for every $S$-scheme $T$. 
\end{myremark}

Above the relative Picard functor lives a functor of effective divisors.

\begin{mydef}
The $\Div_{X/S}$ functor sends $T$
to 
\[
\Div_{X/S} ( T ) 
= 
\{ 
D \subset X \times_S T 
\text{ effective divisor that is flat over $T$}
\}
\]	
\end{mydef}

\begin{mydef}
The linear system functor is the intermediate piece:
it sends $h : T \to S $ to 
\begin{align*}
& \LinSys_{ \mathscr L / X / S } ( T )  \\
& = 
\{
\text{relative effective divisors $D$ }
\mid 
\OOO_{X\times_S T} ( D ) 
\simeq h^* \mathscr L \otimes f_T^* \mathscr N \\
& 
\qquad \text{ for some invertible sheaf $\mathscr N $ on $T$}
\}
\end{align*}
\end{mydef}

In the situation we are considering in this paper, the assumptions of the following two theorems hold.

\begin{mythm}[{\cite[Theorem 3.7]{kleiman2005picard}}]
	Assume that $X\to S$ is projective and flat.
	Then the functor
	$\Div_{X/S}$
	is representable by an open subscheme $\DDiv_{X/S}$ of the Hilbert scheme of $X$.
\end{mythm}

\begin{mythm}[{\cite[Theorem 4.8]{kleiman2005picard}}]
	\label{thm-existence-criterion-picard}
	Assume that $X \to S $ is projective Zariski locally over $S$, and is flat with integral geometric fibres. 
	Then $\Pic_{X / \kb}$
	is representable by a $\kb$-scheme 
	$\PPic_{X / \kb}$.
\end{mythm}

\begin{mydef}[{\cite[Definition 4.6]{kleiman2005picard}}]
    The Abel 
    (or Abel-Jacobi) map
    is the map of functors 
    \[
    \Div_{X/S} ( T ) 
    \to \Pic_{X/S} ( T )
    \]
    sending a relative effective divisor $D$ on $X_T / T$
    to its associated sheaf $\mathscr O_{X_T} ( D ) $. 
    In case these functors are representable, 
    then there is a corresponding morphism of schemes 
    \[
    \AJ_{X/S}
    :
    \DDiv_{X/S}
    \to 
    \PPic_{X/S} . 
    \]
\end{mydef}

\subsection{Poincaré sheaves for the Picard scheme}

From now on we assume that $\PPic_{X/S} $ exists.

\begin{mydef}
	A \emph{Poincaré sheaf} 
	$\mathscr P$
	is an invertible sheaf on $ X \times \PPic_{X/S}$ 
	satisfying the following universal property:
	for any $g : T\to S$
	and any invertible sheaf $\mathscr L$ on $Y_T$ 
	there exists a unique 
	\[
	h : T \to  \PPic_{X/S} 
	\]
	such that 
	\[
	\mathscr L 
	\simeq 
	( \mathrm{id}_Y , h )^* \mathscr P
	\otimes_{\mathscr O_{Y_T}} g_T^* \mathscr N
	\]
	for some invertible sheaf $\mathscr N$ on $T$.
\end{mydef}

\begin{myremark}
	If $\PPic_{X/S}$ exists as a scheme, 
	then a Poincaré sheaf exists: it is given by 
	$\mathrm{id} : \PPic_{X/S} \to \PPic_{X/S}$,
	viewed as a $\PPic_{X/S}$-point 
	of $\PPic_{X/S}$ and corresponding to the isomorphism class of a certain invertible sheaf on $X \times_S \PPic_{X/S}$.
\end{myremark}

By \cite[7.7.6]{EGA3-II} if $f ' : X ' \to S ' $ is a proper morphism 
with $S'$ locally Noetherian
and 
$\mathscr F$ is a coherent $\mathscr O_{X'}$-module,
assumed to be flat on $S'$,
then 
there exists a coherent $\mathscr O_{S'}$-module $\mathscr Q$ 
on $S'$
and an isomorphism of functors in the quasi-coherent $\OOO_{S'}$-module $\mathscr M$
\[
\mathscr{H}om_{\OOO_{S'}}
( \mathscr Q , \mathscr M )
\longrightarrow 
f'_* ( \mathscr F \otimes_{\mathscr O_{X'}} (f')^* \mathscr M )
\]
inducing in particular an isomorphism of functors
\[
\Hom_{\OOO_{S'}}
( \mathscr Q , \mathscr M )
\longrightarrow 
H^0 ( X ' , \mathscr F \otimes_{\mathscr O_{X'}} (f')^* \mathscr M ) .
\]
In particular, 
if $\mathbf V ( \mathscr Q )$ denotes the \emph{total space of $\mathscr Q$},
that is to say the scheme associated to the symmetric $\mathscr O_{S'}$-algebra $\mathscr{S}\mathit{ym}_{\mathscr O_{S'}} ( \mathscr Q )$,
one gets an isomorphism 
\begin{equation}
\label{equation-caracterisation-of-V(Q)}
H^0 ( X' \times_{S'} T' ,  \mathscr F \otimes_{\OOO_{S'}} \OOO_{T'} ) 
\overset{\sim}{\longrightarrow}
\Hom_{\OOO_{S'}} ( \mathscr Q , \mathscr O_{T'} )
\overset{\sim}{\longrightarrow}
\Hom_{S'} ( T ', \mathbf V_{S'} ( \mathscr Q ) )
\end{equation}
functorial in the scheme $T' \to S'$ corresponding to the $\mathscr O_{S'}$-module $\mathscr M = \mathscr O_{T'} $. 

\begin{myremark}
In general, one would like to apply this to 
a Poincaré sheaf on 
\[
X ' = X \times_\kb \PPic_{X/S}
\]
relatively to $S ' = \PPic_{X/S}$
but there is \textit{a priori} no reason for a universal sheaf 
on $X \times_\kb \PPic_{X/S}$ to be flat above $ \PPic_{X/S}$.
This is actually known to be false: if $X = \CCC$ is a curve of genus $g>1$,
above the component parametrising invertible sheaves of degree $n=2g-2$
one can show 
via Serre duality
that for a point $x \in \PPic_{\CCC / \kb}^{2g-2}$ we have 
\[
h^0 ( \CCC \times \{ x \} , \mathscr P_x ) = n + 1 - g = g - 1.
\]
if $\mathscr P_x $ is not isomorphic to $\omega_\CCC$
while 
\[
h^0 ( \CCC \times \{ x \} , \omega_\CCC ) = n + 2 - g = g 
\]
if $\mathscr P_x \simeq \omega_\CCC$,
thus contradicting flatness of the push-forward of $\mathscr P$ to $\PPic_{\CCC / \kb}$. 
\end{myremark}

So the claim from \cite[Exercise 4.7]{kleiman2005picard} can be corrected using the flattening stratification from
 \cref{lemma-flattening-stratification-of-finitely-presented-sheaves}. 
Here, we write $\mathscr Q$ for a collection of coherent sheaves on a stratification of $\PPic_{X / S}$
obtained by flattening the push-forward of the Poincaré sheaf $\mathscr P$.  
Then, \eqref{equation-caracterisation-of-V(Q)} specialises to
\begin{equation}\label{equation-caracterisation-of-V(Q)-specialised}
H^0 ( X \times_S T ,  \mathscr L )
\overset{\sim}{\longrightarrow}
\Hom_S ( T , \mathbf V_{\PPic_{X / S }} ( \mathscr Q ) ) 
\end{equation}
factorially in $T \to \PPic_{X / \kb}$,
assuming that it 
factorises through one of the strata and
corresponds to an invertible sheaf $\mathscr L$ on $X \times_S T$.  

\begin{myptn}[{\cite[Exercise 4.7]{kleiman2005picard}, corrected}]
\label{proposition-divisors-as-proj}
Assume that a Poincaré sheaf $\mathscr P$ exists.
	Then $\AJ_{X/S} : \DDiv_{X/S} \to  \PPic_{X/S}$
	represents the functor 
	\[
	\LinSys_{\mathscr P / ( X \times_S \PPic_{X/S} ) / \PPic_{X/S}}. 
	\]
As schemes piecewisely defined above $\PPic_{X/S}$, 
\[
( \PP ( \mathscr Q ) \to \PPic_{X/S} ) 
\simeq 
 ( \AJ_{X/S}
    :
    \DDiv_{X/S}
    \to 
    \PPic_{X/S} ) .
\]
\end{myptn}

\begin{myexample}
In the case of our curve $\CCC$, 
since $H^2 ( \CCC , \OOO_\CCC ) = 0$, 
the Picard scheme $\PPic_{\CCC / \kb}$ is smooth. 
	Let us assume moreover 
that our smooth projective curve $\CCC$ 
admits a $\kb$-divisor $D_1$ of degree one. 
Then 
\[
\PPic_{\CCC / \kb}
=
\coprod_{d\in \ZZ} \PPic_{\CCC / \kb}^d 
\]
where each $\PPic_{\CCC / \kb}^d$ is isomorphic to $\PPic_{\CCC / \kb}^0$
via $[\LLL]\mapsto [\LLL \otimes \mathcal O_\CCC ( - d D_1 )]$. 

When $d>2g-2$, then the restriction of $\mathscr P$ to $\PPic_{\CCC / \kb}^d$ is locally free of rank $n-g+1$. 
It is not necessarily locally free for small degrees
but for each $d \leqslant 2g-2$ we can still find a stratification of $\Pic^0 ( \CCC )$
into locally closed subsets 
such that on each stratum, the restriction of $\PPP$ is locally free,
by \cref{lemma-flattening-stratification-of-finitely-presented-sheaves}. 
\end{myexample}

\subsection{Motivic Euler products on $\CCC$ as motivic functions on $\DDiv_{\CCC / \kb}$}

In this last subsection, let us assume for a moment that $ \XXX =  ( X_m )_{m\in M}$
is a family of elements of $\KVar{\CCC}$ or $\MMM_\CCC$
indexed by $M = \NN^* = \NN \setminus \{ 0 \}$. 
It is conceptually important but also convenient for us to interpret the entire family of coefficients (minus the constant)
of the motivic Euler product on $\CCC$
associated to the family $\XXX$  
\[
\prod_{x\in \CCC / \kb}
\left ( 
1 
+ 
\sum_{m\in \NN^*} X_m 
\right )
=
\sum_{\mu} 
	\left 
	[
	\Sym_{\CCC / \kb}^\mu ( \mathscr X )_* 
	\to \Sym_{/ \kb}^\mu ( \CCC )_*  
	\right 
	]
	\TT^\pi 
\]
as a motivic function $\Sym_{\CCC / \kb}^\bullet ( \mathscr X )_*$  
on the scheme $\DDiv_{\CCC / \kb}$ of effective divisors on $\CCC$, via the identification
\[
\DDiv_{\CCC / \kb}^d 
\simeq 
\Sym_{/ \kb}^d ( \CCC ) 
\]
available for every $d\in \NN^*$,
see Proposition 6.3.9 of \cite[p.~437]{deligne1973cohomologie}.
In particular, 
it is compatible with the addition law:
for every $d_1 , d_2 \in \NN^*$ such that $d_1 + d_2 = d$ 
the diagram 
\[
\begin{tikzcd}
\DDiv_{\CCC / \kb}^{d_1} \times_\kb \DDiv_{\CCC / \kb}^{d_2} \rar \dar["+"] & \Sym_{\kb}^{d_1} ( \CCC ) \times_\kb \Sym_{\kb}^{d_2} ( \CCC ) \dar \\
\DDiv_{\CCC / \kb}^d \rar & \Sym_{\kb}^d ( \CCC )
\end{tikzcd}
\]
is commutative 
by \cite[Lemma 6.3.9.1]{deligne1973cohomologie}.

Having this in mind, it turns out that the relation 
\[
\prod_{x\in \CCC / \kb}
\left ( 
1 
+ 
\sum_{m\in \NN^*} X_m 
\right )
\prod_{x\in \CCC / \kb}
\left ( 
1 
+ 
\sum_{m\in \NN^*} X_m ' 
\right )
=
\prod_{x\in \CCC / \kb}
\left ( 
\left ( 
1 
+ 
\sum_{m\in \NN^*} X_m 
\right )
\left ( 
1 
+ 
\sum_{m\in \NN^*} X_m '
\right )
\right )
\]
can be interpreted 
as a relation between two motivic functions on $\DDiv_{\CCC / \kb}$:
\[
\left [
\begin{tikzcd}
	\Sym_{\CCC / \kb}^\bullet ( \mathscr X )_* \boxtimes \Sym_{\CCC / \kb}^\bullet ( \mathscr X ')_* \dar \\
	\DDiv_{\CCC / \kb } \times_\kb \DDiv_{\CCC / \kb } \dar["+"]\\
	\DDiv_{\CCC / \kb }
\end{tikzcd}
\right ]
=
\left [
\begin{tikzcd}
\Sym_{\CCC / \kb}^\bullet ( \mathscr X \boxtimes \mathscr X ' )_* \dar \\
	\DDiv_{\CCC / \kb } 
	\end{tikzcd}
\right ]
\]
This operation is compatible with the tensor product on the associated invertible sheaves thanks to the commutativity of the diagram
\begin{equation}
	\begin{tikzcd}
\DDiv_{\CCC  / \kb } \times_\kb \DDiv_{\CCC / \kb } \rar["+"] \dar["{(\AJ_{\CCC / \kb} , \AJ_{\CCC / \kb})}"] & 	\DDiv_{\CCC  / \kb } \dar["\AJ_{\CCC / \kb}"] \\
\PPic_{\CCC / \kb } \times_\kb \PPic_{\CCC / \kb } 
\rar["\otimes"] & \PPic_{\CCC / \kb } .
\end{tikzcd}
\end{equation}
Given a non-empty finite set $\mathfrak I$, this interpretation can be easily generalised to $\CCC$-families
indexed by $\NN^\mathfrak I \setminus \{ \mathbf 0 \}$.

%%%%%%%%%%%%%%%%%%%%%%%%%%%%%%%%%%%%%%%%%%%%%%%%
%%%%%%%%%%%%%%%%%%% SECTION %%%%%%%%%%%%%%%%%%%%
%%%%%%%%%%%%%%%%%%%%%%%%%%%%%%%%%%%%%%%%%%%%%%%%

\section{Motivic Batyrev-Manin-Peyre principle}

\label{section-motivic-BMP}

In this section we recall the statements
of the motivic Batyrev-Manin-Peyre principle 
from \cite{faisant2023motivic-distribution},
adapting it to our simpler isotrivial setting.

\subsection{Motivic Tamagawa numbers}

The following Lemma is a consequence of \cref{lemma:convergence-criterion-Z-coeff-multivariable}. 

\begin{mylemma}
	\label{lemma-convergence-motivic-tamagawa-dimension}
	Let $V$ be a $\kb$-variety such that 
	\[
	\left ( 1 - \LL_\kb \right )^{\rk ( \Pic ( V  )) }
	\frac{[V]}{\LL_\kb^{\dim ( V )}} 
	\in 
	1 + \mathcal F^{-2} \MMM_\kb. 
	\]
	Then the motivic Euler product 
	\[
	\prod_{p\in \CCC}
	\left ( 1 - \LL_p^{-1} \right )^{\rk ( \Pic ( V  )) }
	\frac{[V\times_\kb \kappa ( p ) ]}{\LL_p^{\dim ( V )}}
	\]
	converges in $\widehat{\mathscr M_\kb}^{\dim}$.
\end{mylemma}

\begin{mydef}
	Assume that $\CCC$ admits a $\kb$-divisor of degree one
	and that $V$ satisfies the assumption of the previous lemma.
	The motivic Tamagawa number of $V \times_\kb \CCC \to \CCC$ is defined as 
    \[
    \tau ( V / \mathscr C )
    =
    \left ( \res_{t=\LL^{-1}} Z^\Kapr_\CCC ( t ) \right )^{\rk ( \Pic ( V  )) }
 	\prod_{p\in \CCC}
	\left ( 1 - \LL_{\kappa ( p )}^{-1} \right )^{\rk ( \Pic ( V  )) }
	\frac{[V\times_\kb \kappa ( p ) ]}{\LL_{\kappa ( p )}} 
    \]
    where $ Z^\Kapr_\CCC ( t )$ is the Kapranov zeta function of $\CCC$ defined as 
    \[
     Z^\Kapr_\CCC ( t )
     =
     \sum_{d\in \NN} \left [
     \Sym_{/\kb}^d ( \CCC ) 
     \right ] t^d . 
    \]
\end{mydef}

\begin{myremark}
	Compared to the definition given in \cite{faisant2023motivic-distribution} 
	we normalised by a factor $\mathbf L_\kb^{\dim ( V ) ( 1 - g ( \CCC )  )}$. 
\end{myremark}

\subsection{Motivic stabilisation of curves}

It is well-known 
\cite{debarre2001higher}
that
for every ${\mdeg \in \Pic ( V )^\vee}$
 the dimension of 
$ \HHom^\mdeg_\kb ( \CCC , V ) $
is bounded bellow by the quantity
\[
\mdeg \scdot \omega_V^{-1} + ( 1 - g )\dim ( V ) 
\]
where $g$ is the genus of $\CCC$.

\medskip 

Let $V$ be a (geometrically irreducible) variety over $\kb$ 
such that 
\begin{itemize}
	\item $ V ( \kb ( \CCC ) ) $ is Zariski-dense in $V$; 
\item both cohomology groups $H^1 ( V , \OOO_V ) $ and $ H^2 ( V , \OOO_V ) $ are trivial;
\item the Picard group of $V$ coincides with its geometric Picard group and has no torsion;
\item its geometric Brauer group is trivial;
\item the class of the anticanonical line bundle of $V$ lies in the interior of the effective cone $\CEff ( V ) $, which itself is rational polyhedral. 
\end{itemize}
In this article, it will be enough to assume that $V$ satisfies the assumption of \cref{lemma-convergence-motivic-tamagawa-dimension}. 
In general, one needs a filtration coming from the weight in mixed Hodge theory, see for example \cite[\S 3.1]{faisant2023motivic-distribution}. 

\begin{myquest}[Motivic Batyrev-Manin-Peyre \cite{faisant2023motivic-distribution} -- isotrivial case]
\label{question-motivic-BMP-isotrivial}
	Assume that $\CCC$ admits a $\kb$-divisor of degree one.
	Do we have
\begin{equation}
	\left [
	 \HHom^\mdeg_\kb ( \CCC , V ) 
	\right ]
	\LL_\kb^{- \mdeg \scdot \omega_V^{-1} - ( 1 - g )\dim ( V )  } 
	\longrightarrow 
	 \tau ( V / \mathscr C )
\end{equation}	
	as $\mdeg \in \Pic ( V )^\vee \cap \CEff ( V ) ^\vee$
	goes arbitrarily far away from the boundary of 	$ \CEff ( V ) ^\vee$ ?
\end{myquest}

%%%%%%%%%%%%%%%%%%%%%%%%%%%%%%%%%%%%%%%%%%%%%%%%
%%%%%%%%%%%%%%%%%%% SECTION %%%%%%%%%%%%%%%%%%%%
%%%%%%%%%%%%%%%%%%%%%%%%%%%%%%%%%%%%%%%%%%%%%%%%

\section{Cox rings and universal torsors of Mori Dream Spaces}

Starting from this section 
and until the end of the article, $X$ is a smooth, projective and geometrically irreducible variety defined over an absolute base field $\kb$,
and we assume that $X$ is a Mori Dream Space. 
We fix once and for all $r$ line bundles $\mathcal L_1 , ... , \mathcal L_r$ on $X$, whose classes generate $\Pic ( X )$. 

\smallskip 

\label{section-cox-rings-MDS}

\subsection{Cox ring of a Mori Dream Space}

We recall some facts about Cox rings, following \cite{hassett-tschinkel2004universal}
and 
\cite[\S 1.1]{bourqui2009eclate3alignes}.
For simplicity, 
we assume that the Picard group of $X$ is 
free of finite rank $r$
and split.
Assuming that $X$ is a Mori Dream Space means 
that its Cox ring 
\[
\Cox ( X )
= \bigoplus_{( n_1, \dots , n_r ) \in \NN^r}
H^0 ( X , \mathcal L_1^{\otimes n_1 } \otimes ... \otimes \mathcal L_r^{\otimes n_r} )
\]
is of finite type,
with an 
ideal of relations 
denoted by 
$\mathcal I_X$.

We fix a finite family
$( s_i )_{i\in \mathfrak I}$
of generators of $\Cox ( X )$
and $(\mathcal D_i)_{i\in \mathfrak I}$
the corresponding set of divisors on $X$.
Let $U$ be the complement in $X$ of the union of the $\mathcal D_i$'s
and $N_X = \kb [U] ^ \times / \kb^\times $. 

The associated $\Pic ( X )$-grading on the $\kb$-algebra 
$\kb [ ( s_i )_{i\in \mathfrak I}]$
is given by 
\[
\deg ( s_i ) = [ \mathcal D_i ] \in \Pic ( X ) ,
 i\in \mathfrak I.
\]
Let $\mathcal I_X$
be the $\Pic ( X )$-homogeneous ideal of $\kb [ ( s_i )_{i\in \mathfrak I}]$ 
\[
\mathcal I_X
= \ker 
\left ( \kb [ ( s_i )_{i\in \mathfrak I}] \twoheadrightarrow \Cox ( X ) 
\right ) 
\]
encoding the relations between the generators. 
From now on we now see the Cox ring of $X$ as
\[
\Cox ( X ) 
\simeq 
\kb [ ( s_i )_{i\in \mathfrak I}]
/
\mathcal I_X . 
\]
There is a fundamental exact sequence 
\[
0 
\longrightarrow
\kb^\times 
\longrightarrow 
\kb [U] ^ \times 
\longrightarrow 
\bigoplus_{i\in \mathfrak I}
\ZZ \mathcal D_i 
\longrightarrow
\Pic ( X )
\longrightarrow 
0
\]
providing the short exact sequence
\begin{equation}
\label{exact-sequence-MDS-Pic}
0 
\longrightarrow
N_X
=
\kb [ U ]^\times / \kb^\times 
\overset{\gamma}{\longrightarrow}
\bigoplus_{i\in \mathfrak I}
\ZZ \mathcal D_i 
\underset{\mathcal D_i \mapsto [ \mathcal D_i ]}
\longrightarrow
\Pic ( X )
\longrightarrow 
0
\end{equation}
which encodes the linear relations between the 
$\mathcal D_i 's $.
This sequence generalises 
the classical exact sequence 
\[
0 
\longrightarrow
\mathcal X^* ( U )  
\overset{\gamma}{\longrightarrow}
\PL ( \Sigma ) 
\longrightarrow
\Pic ( X_\Sigma )
\longrightarrow 
0
\]
whenever 
$X = X_\Sigma$
is a projective toric variety defined by a complete fan $\Sigma$,
$\mathfrak I = \Sigma ( 1 ) $
is the set of rays of the fan, $\mathcal X^* ( U )$ is the set of characters of the dense open torus 
and $\PL ( \Sigma ) $ 
is the space of piecewise linear functions on the fan $\Sigma$.

\medskip 

In general, since $\Pic ( X ) $
and $\bigoplus_{i\in \mathfrak I}
\ZZ \mathcal D_i$
are free of finite rank, 
by Rosenlicht's lemma,
so is $N_X$.
Moreover, the effective cone of $X$ is the image in $\Pic ( X ) \otimes \RR$
of the cone
$\oplus_{i\in \mathfrak I} \RR_+ D_i$.

\begin{myexample}
As recalled in the introduction, toric varieties are exactly
the Mori Dream Spaces whose Cox ring is a polynomial ring \cite{cox1995homogeneous}.
\end{myexample}

\begin{myexample}
	In the setting of Manin's conjecture in positive caracteristic, 
	varieties with Cox ring admitting exactly one linear relation,
	called \emph{linear intrisinc hypersurfaces},
	have been considered by Bourqui in  
	\cite{bourqui2012linear-intrinsic}. 
	A family of \emph{intrisinc quadrics}
	is considered in \cite{bourqui2011quadriquesintrinseques}. 
\end{myexample}

\begin{myexample}
	Among del Pezzo surfaces, 
	the degree five split del Pezzo surface 
	is the first non-toric surface obtained by successively blowing-up points of $\PP^2_\kb$ in general position.
	Its Cox ring 
	is isomorphic to the quotient 
	\[
	\kb [ a_1 , a_2 , a_3 , a_4 , a_{12} , a_{13} , a_{14} , a_{24} , a_{23} , a_{34} ] / \mathscr I_{\mathrm{DP}_5} 
	\]
	where $\mathscr I_{\mathrm{DP}_5} $ is the ideal generated by the elements
	\begin{equation}
    \begin{aligned}
        &a_4a_{14}-a_3a_{13}+a_2a_{12},\\
        &a_4a_{24}-a_3a_{23}+a_1a_{12},\\
        &a_4a_{34}-a_2a_{23}+a_1a_{13},\\
        &a_3a_{34}-a_2a_{24}+a_1a_{14},\\
        &a_{12}a_{34}-a_{13}a_{24}+a_{23}a_{14}.
    \end{aligned}  
\end{equation}
\end{myexample}

\begin{myexample}
In general, for del Pezzo surfaces, the less the degree, the more complex is the Cox ring. 
For example, smooth \emph{split} cubic surfaces 
have a Cox ring generated by a minimal set of $27$ elements (corresponding to the lines)
satisfying $81$ algebraic relations. 
\end{myexample}

\subsection{Universal torsors and Cox rings}

Assume that $\mathcal L_1 , ... , \mathcal L_r$ are $r$ semi-ample line bundles  on $X$. 
If $\mathcal L$ is a line bundle, 
let $\mathcal L^\times$ be the $\GG_m$-torsor obtained from $\mathcal L$ by removing the zero section. 
The $T_{\NS}$-torsors
\begin{equation}
	\mathcal T_{\underline{\mathcal L}}
	= \begin{tikzcd}
		( \mathcal L_1^{-1})^\times \times_X ... \times_X (\mathcal L_r^{-1})^\times \dar \\
		X 
	\end{tikzcd}
\end{equation}
is a universal torsor on $X$. 
Take the Cox ring of $X$ to be 
\[
\Cox ( X )
= \bigoplus_{n_1, ... , n_r \in \NN^r}
H^0 ( X , \mathcal L_1^{\otimes n_1 } \otimes ... \otimes \mathcal L_r^{\otimes n_r} ) .
\]

\begin{mythm}[{\cite[Theorem 5.7]{hassett2009rationalsurfaces}}]
There is a 
$T_{\NS}$-equivariant open
embedding 
\[
\mathcal T_{\underline{\mathcal L}} \hookrightarrow \Spec ( \Cox ( X )) .
\]
\end{mythm}

% TODO [URGENT] Détailler cet embedding 

Under this embedding,
there exists a certain finite set $B_X$ of subsets of $I$ such that $\mathcal T_{\underline{\mathcal L}}$,
via the isomorphism
\[
\Spec ( \Cox ( X ) ) \overset{\sim}{\longrightarrow} \Spec ( \kb [ ( s_i )_{i\in \mathfrak I} ] / \mathcal I_X ) ,
\]
is identified to the complement in $\Spec ( \Cox ( X ))$ of 
\begin{equation}
\bigcup_{\mathfrak I' \in B_X} 
\left \{ 
\prod_{i\in \mathfrak I'} s_i = 0 
\right \}. 
\end{equation}

% TODO [URGENT] Torseurs universels : remplir cette section 

\subsection{Collections and functor of points of an MDS}

\begin{mydef}[$X$-Collection]
Let $S$ be a scheme.
An $X$-collection ($N_X$-trivialized) over $S$ is the datum 
$( ( \mathcal L_i , u _i )_{i\in \mathfrak I} ,  ( c_\nn )_{\nn \in N_X} )$ 
of 
\begin{itemize}
	\item a family $( \mathcal L_i )_{i\in \mathfrak I}$ of invertible sheaves on $S$ together with global sections $(u_i )_{i\in \mathfrak I}$,
	\item an $N_X$-trivialisation of $(\mathcal L_i )_{i\in \mathfrak I}$, that is to say a family of isomorphisms 
	\[
	c_\nn : \otimes_{i\in \mathfrak I} \mathcal L^{n_i} \longrightarrow \mathcal O_S 
	\]
	indexed by $N_X$ such that 
	\[
	c_\nn \otimes c_{\nn '} = c_{\nn + \nn ' }
	\]
	for all $\nn , \nn ' $ in $N_X$,
\end{itemize}	
respecting the following two additional conditions:
\begin{itemize}
	\item the induced morphism
	\[
	\bigoplus_{\mathfrak I' \in B_X} 
	\bigotimes_{i\in \mathfrak I'}
	\mathcal L_i^{-1} 
	\longrightarrow 
	\mathcal O_S
	\]
	is surjective,
	\item and for all homogeneous elements $F$ of $\mathcal I_X$,
	\[
	F ( u_i ) = 0.
	\]
\end{itemize}
An isomorphism 
between two $X$-collections 
\[
( ( \mathcal L_i , u _i )_{i\in \mathfrak I} ,  ( c_\nn )_{\nn \in N_X} )
\]
and
\[
( ( \mathcal L_i ' , u _i ' )_{i\in \mathfrak I} ,  ( c_\nn ' )_{\nn \in N_X} )
\]
is a family of isomorphisms 
\[
\mathcal L_i \overset{\sim}{\longrightarrow} \mathcal L_i ' 
\]
sending $u_i $ to $u_i '$
and $c_\nn$ to $c_{\nn'}$.

We say that a collection is non-degenerate when all the $u_i$'s are non-zero.

Let $\mathcal C_X$ be the functor
on $\mathrm{Sch}_\kb$
of isomorphism classes of $X$-collections.
\end{mydef}

\begin{myremark}
As remarked in \cite[Remarque 1.7]{bourqui2009eclate3alignes}, 
two $N_X$-trivialisations above $S$ 
of a given familly of line bundles
differ from a unique group morphism $N_X \to H^0 ( S , \mathcal O_S)^\times$. 
In particular, if $S$ is projective,
two $N_X$-trivialisations on $S$ only differ from a multiplicative constant. 
\end{myremark}

\begin{mydef}[Universal collection on $X$]
The universal collection on $X$
is the collection 
\[
\left 
( 
(
\mathcal O_X ( \mathcal D_i ) , s_i )_{i\in \mathfrak I} , ( c_\nn )_{\nn \in N_X} )
\right )
\]	
where $(c_\nn)$ is the trivialisation of $(\mathcal O ( \mathcal D_i ))_{i\in \mathfrak I}$
induced by the exact sequence \eqref{exact-sequence-MDS-Pic}.
\end{mydef}

\begin{mythm}[{\cite{cox1995functor,bourqui2009eclate3alignes}}]
\label{thm-cox-bourqui-fonctor-of-points-MDS}
    The applications 
    \[
    \begin{array}{rll}
      X ( S ) 
    & \longrightarrow & 
    \mathcal C_X ( S )\\
    ( h : S \to X ) & \longmapsto & h^* \left 
( 
(
\mathcal O_X ( \mathcal D_i ) , s_i )_{i\in \mathfrak I} , ( c_\nn )_{\nn \in N_X} )
\right )
    \end{array}
    \quad 
    S \in \mathrm{Sch}_\kb,
    \]  
    induce an isomorphism of functors. 
\end{mythm}

\section{Parametrisation of curves on Mori Dream Spaces}

\label{section-parametrisation-curves-MDS}

Again, let $X$ be a smooth projective and geometrically irreducible variety defined above an absolute base field $\kb$,
which is a Mori Dream Space.
It is natural to see 
the scheme 
\[
\HHom_\kb ( \CCC , X )_U 
\]
parametrising morphisms 
$f : \CCC \to X$
intersecting the open subset 
\[
U
= X - \bigcup_{i\in \mathfrak I} \mathcal D_i 
\]as a disjoint union
indexed by $\dd \in \NN^I$
of schemes 
\begin{equation}
\label{height-function-I}
\HHom_\kb^d ( \CCC , X )_U
\longrightarrow 
\Sym_{/\kb}^\dd ( \CCC )  
\end{equation}
where 
\[ 
\Sym_{/\kb}^\dd ( \CCC )  
= 
\prod_{i\in \mathfrak I} \Sym_{/\kb}^{d_i} ( \CCC )
\]
parameterises $\mathfrak I$-tuples of effective divisors on $\CCC$,
of degrees given by $\dd = (d_i)_{i\in \mathfrak I}$,
and 
the morphism 
in \eqref{height-function-I}
sends a morphism ${f : \CCC \to X}$,
whose image intersects the complement $U$ of the $\mathcal D_i$'s,
to the $I$ tuple of effective divisors 
$( f^* \mathcal D_i )_{i\in \mathfrak I}$ on $\CCC$.

\medskip 

In other words,
the fibre of \eqref{height-function-I} 
above any given tuple of effective divisors $(D_i)_{i\in \mathfrak I}$ on the curve $\CCC$
``counts'' the number of morphisms $\CCC \to X$ 
having intersections
with the $\mathcal D_i$'s
 prescribed by this tuple $(D_i)_{i\in \mathfrak I}$. 

\begin{center}
\begin{minipage}{\linewidth}% to keep image and caption on one page
\makebox[\linewidth]{%        to center the image
  \includegraphics{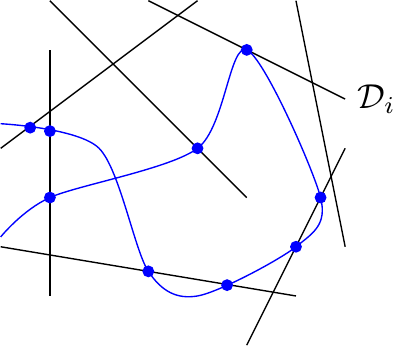}
  }
\end{minipage}
\end{center}

The parametrisation involves three types of constraints, ordered by decreasing level of complexity:
\begin{enumerate}
		\item algebraic equations coming from $\mathscr I_X$ that $\mathfrak I$-tuples of sections of certain invertible sheaves on the curve have to satisfy;
	\item coprimality conditions on the divisors of these sections, encoded at the level of the scheme of effective divisors of the curve, see \cref{def-coprime-condition-divisors} \cpageref{def-coprime-condition-divisors};
	\item linear equivalence relations between these divisors, encoded at the level of the Picard group scheme of the curve, see \cref{def-linear-equiv-cond-on-tuples-of-divisors} \cpageref{def-linear-equiv-cond-on-tuples-of-divisors}.
\end{enumerate}

\subsection{$X$-collections on $\CCC$}

If one replaces $S$ by $\CCC \times_\kb S$
in \cref{thm-cox-bourqui-fonctor-of-points-MDS} 
one gets
an isomorphism of functors between 
\[
S \rightsquigarrow \HHom_\kb ( \CCC , X )_U ( S ) 
\]
and 
the subfunctor of 
\[
S \rightsquigarrow \mathcal C_X ( \CCC \times_\kb S )
\]
of isomorphism classes of \emph{non-degenerate $X$-collections}.
An $S$-point of $\HHom_\kb ( \CCC , X )_U$, 
corresponding to a morphism 
\[
f_S : \CCC \times_\kb S \to X,
\]
seen as an $(\CCC \times_\kb S)$-points of $X$,
is sent to the (isomorphism class of) the $X$-collection on $\CCC \times_\kb S$ given by 
\[
(     f_S^* \mathcal O ( \mathcal D_i )
,
f_S^* u_i , f_S^* ( c_\nn ) 
) .
\]
Conversely,
if $( \mathcal L_i , v_i  , ( c_\nn ' )) $
is an $X$-collection on $S ' = \CCC \times S$,
then $f_S$ 
is the morphism sending $(x,s)$
to $v_i ( ( x , s) )_{i \in I} \in  \mathbf A_\kb^\mathfrak I ( S ) $.
To be precise, this depends on the choice of a local trivialisation of the $\mathcal L_i$'s, but two different choices lead to coordinates that are related by an element of $T_{\NS}$, 
so $v_i ( ( x , s) )_{i \in I}$ is defined up to the action of $T_{\NS}$ and will give a unique point in $X$.
Indeed, the vanishing condition ensures that it is actually in $\Spec ( \kb [ ( s_i )_{i\in \mathfrak I} ] / \mathcal I_X ) ( S )$. 
The coprimality conditions
ensure that 
\[
( v_i ( x , s ) )_{i \in I} \in \mathcal T_X \hookrightarrow \Spec ( \Cox ( X ) )
\]
via the inverse of the identification  
\[
\Spec ( \Cox ( X ) ) \overset{\sim}{\longrightarrow} \Spec ( \kb [ ( s_i )_{i\in \mathfrak I} ]/ \mathcal I_X ) 
\]
we fixed once and for all,
which means that one can compose by the $T_{\NS}$-torsor ${\pi : \mathcal T_X \to X}$
to finally land in $X$. 

\[
\begin{tikzcd}
	&   \mathcal T_X \rar[hook] \dar & \Spec ( \Cox ( X ) ) \rar["\sim"] & \Spec ( \kb [ ( s_i )_{i\in \mathfrak I} ] / \mathcal I_X ) \rar[hook] & \mathbf A_\kb^\mathfrak I  \\
	\CCC \urar[dashed] \rar & X &  
\end{tikzcd}
\]

%\begin{myremark}
%Uniqueness of the lifts 
%can be obtained by marking points on the $\LLL_i$'s 
%and on 
%$\mathbf A_\kb^\mathfrak J \setminus \cup_i \{ s_i = 0 \}$
%or equivalently by considering rigidified line bundles. 

%Let us fix a point $x \in \CCC ( \kb ) $. 
%For any collection $( \mathcal L_i , v_i  , ( c_\nn ' )) $ on $\CCC \times_\kb S$
%and any rigidification 
%${\alpha_i : x^* \mathcal L_i \overset{\sim}{\to} \mathcal O_S}$ 
%there is actually a \emph{unique} possible choice for the  trivialisation 
%$(c_\nn ' )$ in the isomorphism class of the collection 
%which is 
%compatible with these rigidifications,
%that is to say such that $\alpha_i^{\otimes n_i} = ( c_\nn )_{| \{ x \} \times S }$ for all $\nn \in N_X$,
%and a representative of the orbit of  
%$v_i ( ( \scdot , s) )_{i \in I}$ 
%is fixed by taking $\alpha_i ( v_i ( ( x , s) )) = 1 \in \mathcal O_S ( S )$ for all $i \in I$. 
% TODO [URGENT] Unicité du relevé : revoir ça 
%\end{myremark}

\subsection{Multidegrees and linear classes of morphism}
Let $M$ be any finitely and freely generated $\ZZ$-module.
Seeing it as a constant group scheme over $\kb$,
it makes sense to consider the group scheme
\[
\HHom_{\kb-\mathrm{gp.sch.}}
( M_\kb , \PPic_{\CCC/\kb} ) .
\]
\begin{myremark}
The choice of a basis of $M$
induces an isomorphism of $\kb$-group schemes 
\[
\HHom_{\kb-\mathrm{gp.sh.}}
( M_\kb , \PPic_{\CCC/\kb} ) 
\simeq 
\PPic_{\CCC/\kb}^{\rk ( M )}
\]
but there is no canonical choice for such an isomorphism. 
\end{myremark}

\subsubsection{Picard class of a morphism}
Recall that we assume that $\Pic ( X )$ is a free $\ZZ$-module of finite rank. 
Moreover, 
since $X\to \Spec ( \kb )$ is projective,
by \cref{thm-existence-criterion-picard} \cpageref{thm-existence-criterion-picard}
the Picard scheme 
\[
\PPic_{X/\kb}
\]
exists 
and $\PPic_{X/\kb}\simeq \Pic ( X )_\kb$ as $\kb$-group schemes. 

There is a morphism of schemes
\[
\HHom_\kb ( \CCC , X ) \longrightarrow \HHom_{\kb-\mathrm{gp.sch.}} ( \PPic_{X/\kb}, \PPic_{\CCC/\kb} )
\]
sending a point parametrising a morphism $f : \CCC \to X $ 
to the morphism of group schemes over $\kb$
given by 
\[
f^* 
= 
\left (
[
L 
] \in \PPic_{X/\kb}
\longmapsto 
\left [ f^* L \right ]
\in \PPic_{\CCC/\kb} 
\right ) .
\]
This allows us to see $\HHom_\kb ( \CCC , X )$
above 
\[
\HHom_{\kb-\mathrm{gp.sch.}} ( \PPic_{X/\kb} , \PPic_{\CCC/\kb} ) 
\]
which is,
since $\Pic ( X )$
is by assumption 
free of finite rank $r$,
non-canonically isomorphic to a product of $r$ copies of $\PPic_{\CCC/\kb}$.
Composing by the degree map 
\[
\deg : \PPic_{X/\kb}\to \ZZ_\kb 
\]
gives us back the multidegree $\degg ( f )$.

Now, 
\begin{enumerate}
	\item sending a tuple $(D_i)_{i\in \mathfrak I} \in (\DDiv_{\CCC / \kb})^\mathfrak I$
to its image in $( \PPic_{\CCC \kb} )^\mathfrak I$
via the Abel-Jacobi morphism of $\CCC$,
\item and seing $( \PPic_{\CCC \kb} )^\mathfrak I$
as 
\[
\HHom_{\kb-\mathrm{gp}} ( \ZZ_\kb^\mathfrak I , \PPic_{\CCC / \kb} )
\]
(via the choice of the images of the canonical basis of 
$\ZZ_\kb^\mathfrak I $)
\item plus composing with the morphism

\[
  \HHom_\mathrm{gp} ( \PPic_{X / \kb}  , \PPic_{\CCC / \kb}  )
  \longrightarrow 
    \HHom_\mathrm{gp} ( \ZZ_\kb^\mathfrak I , \PPic_{\CCC / \kb}  ) 
\]
induced by \eqref{exact-sequence-MDS-Pic},
\end{enumerate}
one gets a commutative diagram
\[
\begin{tikzcd}
	& \HHom_\kb ( \CCC , X )_U \dlar \drar \dar & \\
  (\DDiv_{\CCC / \kb})^\mathfrak I \arrow[r] 
  & \HHom_\mathrm{gp} ( \ZZ_\kb^\mathfrak I , \PPic_{\CCC / \kb}  )  
  &     \HHom_\mathrm{gp} ( \PPic_{X / \kb}  , \PPic_{\CCC / \kb}  ) \lar 
\end{tikzcd}
\]

\subsubsection{Coprimality conditions}

\begin{mydef}
\label{def-coprime-condition-divisors}
The subscheme 
\[
\Sym^\dd_\kb ( \CCC )_{X,**}
\]
is the open subscheme of $\Sym^\dd_\kb ( \CCC )$
parametrising 
$\mathfrak I$-tuples $(D_i)_{i\in \mathfrak I}$
of effective divisors on $\CCC$ satisfying the coprimality condition:
\[
\forall \, J \in B_X 
\quad
\bigcap_{j\in J} \Supp ( D_j ) = \varnothing .  
\]
When the degrees of the divisors are not specified,
let 
\[
( \DDiv_{\CCC / \kb} )^\mathfrak I_{X,**}
\]
be the corresponding open subscheme of 
$( \DDiv_{\CCC / \kb} )^\mathfrak I$.
\end{mydef}

\subsubsection{Linear conditions on effective divisors}
Recall that we see $N_X$ as sublattice of $\ZZ^I$ via the exact sequence \eqref{exact-sequence-MDS-Pic} given on \cpageref{exact-sequence-MDS-Pic}.
Inside \[
\PPic^\dd_{\CCC / \kb} 
= 
\prod_{i\in \mathfrak I} \PPic^{d_i}_{\CCC / \kb} 
\subset ( \PPic_{\CCC / \kb} )^\mathfrak I
,
\]
let us 
consider the closed subscheme 
\[
\left (
\PPic^\dd_{\CCC / \kb}
\right )_X
\] 
of tuples of classes of invertible sheaves $\mathscr L_i$ satisfying the condition
\[
\forall \, \nn \in N_X \quad \otimes_{i\in \mathfrak I} \mathscr L_i^{\otimes n_i} \simeq \OOO_\CCC .
\]
Of course, if 
\[
\sum_{i\in \mathfrak I} n_i d_i \neq 0 
\]
for some $n\in N_X$,
then $(\PPic^\dd_{\CCC / \kb})_X$ 
is necessarily empty.

Again by \eqref{exact-sequence-MDS-Pic},
inside $\HHom_\mathrm{gp} ( \ZZ_\kb^\mathfrak I , \PPic_{\CCC / \kb}  )$,
our group of Picard classes 
\[
\HHom_\mathrm{gp} ( \PPic_{X / \kb}  , \PPic_{\CCC / \kb}  )
\]
is identified with $\left (
\PPic^\dd_{\CCC / \kb}
\right )_X$. 

\begin{myexample}
    When $X=\PP^1_\kb$,
then $N_X$ is identified with 
$\ZZ ( 1 , - 1 )$ inside $\ZZ^2$ 
and what one obtains 
for $( \PPic^\dd_{\CCC / \kb} )_{\PP^1_\kb}$
is 
the empty scheme if $\dd \notin \ZZ ( 1 , 1 )$
and 
the diagonal of 
$\PPic^d_{\CCC / \kb} \times_\kb \PPic^d_{\CCC / \kb}$ otherwise. 
\end{myexample}

\begin{mydef}\label{def-linear-equiv-cond-on-tuples-of-divisors}
    For any $\dd \in \NN^\mathfrak I$,
    the scheme 
    $\Sym^\dd_\kb ( \CCC )_{X,*}$ is given by the Cartesian square
    \[
    \begin{tikzcd}
        \Sym^\dd_\kb ( \CCC )_{X,*} \rar \dar \arrow[dr, phantom, "\ulcorner", very near start] & \Sym^\dd_\kb ( \CCC ) \dar["\AJ_{\CCC/\kb}"]\\
       ( \PPic_{\CCC / \kb}^\dd)_X \rar &\PPic_{\CCC / \kb}^\dd .
    \end{tikzcd}
    \]
\end{mydef}

\subsubsection{Putting all together}

\begin{mydef}
\label{def-all-conditions-on-divisors}
We set 
\[
\Sym^\dd_\kb ( \CCC )_X
=
\Sym^\dd_\kb ( \CCC )_{X,*}
\cap 
\Sym^\dd_\kb ( \CCC )_{X,**}
\]
for all $\dd\in \NN^I$,
equivalently
\begin{equation}
	( \DDiv_{\CCC / \kb} )^\mathfrak I_X
=
( \DDiv_{\CCC / \kb} )^\mathfrak I_{X,*}
\cap 
( \DDiv_{\CCC / \kb} )^\mathfrak I_{X,**}
\end{equation}
\end{mydef}

\subsection{Parametrisation}
We fix a Poincaré sheaf $\mathscr P$ for $\PPic_{\CCC / \kb}$ 
and a flattening stratification of $\PPic^d_{\CCC / \kb}$ for $\mathscr P$ and $d \in \mathbf N$ with $d \leqslant 2g -1$. 
Working relatively above $\PPic_{\CCC  / \kb} $,
we identify the total space 
$\mathbf V_{\PPic_{\CCC / \kb}} ( \mathscr Q )$
of 
the line bundle $\mathcal O_{\mathbf P ( \mathscr Q )} ( -1 )$ minus the zero section 
with the $\GG_{m,\PPic_{\CCC / \kb }}$-torsor
\[
\Spec_{\mathbf P ( \mathscr Q )}
\left ( \bigoplus_{\bm m \in \ZZ} 
\mathscr O_{\mathbf P ( \mathscr Q )} ( m ) 
\right )
\longrightarrow \mathbf P ( \mathscr Q ) 
\simeq \DDiv_{\CCC /\kb}  
\] 
which we denote by $\mathscr O_{\DDiv_{\CCC /\kb}} ( -1 )^\times $. 
The situation is described by the following diagram of functors.
\[
\begin{tikzcd}
%& & \mathcal O_{ \DDiv_{\CCC / \kb}} ( - 1 )  \dar[equal]\\
\mathcal O_{ \DDiv_{\CCC / \kb}} ( - 1 )^\times \dar \rar[r,equal]
& \Spec_{\PP ( \mathscr Q )} ( \oplus_{m\in \ZZ} \mathscr O_{\PP ( \mathscr Q )} ( m ) ) \dar \ar[r,hook] & 
\mathbf V_{\PPic_{\CCC / \kb}} ( \mathscr Q ) %\Spec_{\PP ( \mathscr Q )} ( \oplus_{m\in \NN} \mathscr O_{\PP ( \mathscr Q )} ( m ) ) 
\\
 \DDiv_{\CCC / \kb} \dar["{\AJ_{\CCC / \kb}}"] \rar[equal] & \PP ( \mathscr Q )  & \\
\PPic_{\CCC / \kb}  & & 
\end{tikzcd}
\] 
When $\CCC = \PP^1_\kb$ this diagram is nothing else than
\[
\begin{tikzcd}
\coprod_{d \in \NN^*} 
\mathcal O_{\PP^d_\kb } ( - 1 ) ^\times  \rar[r,equal]
& \coprod_{d \in \NN^*} \mathbf A_\kb^{d+1} \setminus \{ \mathbf 0 \} \dar["/\GG_m"] \ar[r,hook]
& \coprod_{d \in \NN^*} \mathbf A_\kb^{d+1}  \\
\coprod_{d \in \NN^*} \DDiv_{\PP^1_\kb / \kb}^d \dar["\deg"] \rar[equal] & \coprod_{d \in \NN^*} \PP_\kb^d  & \\
\ZZ_\kb = \coprod_{d\in \ZZ} \Spec ( \kb ) & & 
\end{tikzcd}
\]

Recall that $(\DDiv_{\CCC / \kb } )^\mathfrak I_X$ 
is the subscheme of $(\DDiv_{\CCC / \kb } )^\mathfrak I$ encoding the linearity and coprimality conditions, see \cref{def-all-conditions-on-divisors} \cpageref{def-all-conditions-on-divisors}.

\begin{mydef}
Let $\widehat X$
and $\widetilde X$
be the subschemes respectively
of 
\[
\widehat{M} = 	\prod_{i\in \mathfrak I} \mathcal O_{\PP ( \mathscr Q ) } ( - 1 )^\times 
\]
and 
\[
\widetilde M  = \widehat{M} / T_{\Pic ( X )} 
\]
obtained by pulling-back $\widehat M$
and $\widetilde M$
to the subscheme 
\[
(\DDiv_{\CCC / \kb } ^\mathfrak I)_X
\]
and restricting to points $ ( x_i )_{i\in \mathfrak I} $ in $\widehat X$ or $\widetilde X$
satisfying 
\[
F (  ( x_i )_{i\in \mathfrak I} ) = 0
\]
for every $\Pic ( X )$-homogeneous  $F \in \mathcal I_X$. 
\label{definition-X-hat}
\[
\begin{tikzcd}
\widehat{M} = 	\prod_{i\in \mathfrak I} \mathcal O_{\PP ( \mathscr Q ) } ( - 1 )^\times  \dar  & \arrow[l,hook] \widehat X \dar \dlar \\
			\widetilde M  = \widehat{M} / T_{\Pic ( X )} \dar & \arrow[l,hook] \widetilde X \dlar \dar \\
	\PP ( \mathscr Q )^\mathfrak I\dar &  \arrow[l,hook] (\DDiv_{\CCC / \kb } ^\mathfrak I)_X
 \dar \dlar \\
	(\PPic_{\CCC / \kb} )^\mathfrak I & \arrow[l,hook] 	(\PPic_{\CCC / \kb} )^\mathfrak I _X
\end{tikzcd}
\]
\end{mydef}

\begin{mythm}
\label{thm-parametrisation-finale}
	There is a 
	piecewise isomorphism of schemes over 
	$(\DDiv_{\CCC / \kb } )^\mathfrak I$
	between 
	$
	\HHom_\kb ( \CCC , X )_U
	$
	and 
	$ 
	\widetilde X .
	$ 
\end{mythm}

\begin{proof}
We would like to use \cref{thm-cox-bourqui-fonctor-of-points-MDS},
so it is sufficient to exhibit an isomorphism of 
functors between the functor of $X$-collections on $\CCC$
and the functor of points of $\widetilde X$,
relatively to $( \DDiv_{\CCC / \kb} )^\mathfrak I$.
But going from isomorphism classes of collections 
to points of $\widetilde X \to (\DDiv_{\CCC / \kb } )^\mathfrak I$ is obvious by definition of the associated functors:
one just sends an $S$-collection to the corresponding $S$-points 
of $\PPic_{\CCC / \kb}$ and $\mathbf V_{\PPic_{\CCC / \kb}} ( \mathscr Q )$. 
The non-degeneracy condition ensures that the orbit of the section under $T_{\Pic ( X )} $ is actually in $\widetilde M$. 
The trivialisation ensures that the image is in $(\DDiv_{\CCC / \kb } ^\mathfrak I)_X$. 
	
	Conversely,
	let $h : S \to \widetilde X$ be a morphism of $\kb$-schemes.
	We assume that its composition $h'$ with $\widetilde M \to ( \PPic_{\CCC / \kb} )^{\mathfrak I}$
	has image contained in one of the strata of the flattening stratification for the Poincaré sheaf $\mathscr P$. 
	Then it induces
	an $\mathfrak I$-tuple of $S$-points of $\DDiv_{\CCC / \kb}$ 
	\[
	D_i \in \DDiv_{\CCC / \kb} ( S ) \qquad i \in \mathfrak I
	\]
	lying via the Abel-Jacobi morphism 
	in the linear systems associated to an $\mathfrak I$-tuple of $S$-point of $\PPic_{\CCC / \kb}$
	\[
	\mathscr L_i 
	= ( h_i ' )^* \mathscr P
	 \in \PPic_{\CCC / \kb} ( S ) \qquad i \in \mathfrak I.
	\]
	Then, by definition of the functors associated to all these schemes, 
	in particular \eqref{equation-caracterisation-of-V(Q)-specialised} \cpageref{equation-caracterisation-of-V(Q)-specialised},
	for each $i\in \mathfrak I$, we can pick  
	any lift $S \to \widehat X$ of $h$ and
	\begin{align*}
	& \sigma_{i,S} \in H^0 ( \CCC \times_\kb S ,  \mathscr L_i ) 
		= \mathbf V ( \mathscr Q ) ( S ) 
	 \text{ such that } \sigma_{i,s} \neq 0 \text{ for all } s \in S  \\
	& \text{defining } D_i \in \DDiv_{\CCC /\kb} ( S ) \text{ such that } \OOO_\CCC ( D_i ) \simeq \mathscr L_i 	
	\end{align*}
	and such that the tuple of classes $( [  \mathscr L_i ] )_{i\in \mathfrak I}$,
	seen as a morphism $\ZZ_S^\mathfrak I \to \PPic_{\CCC / \kb} ( S )$,
	is 
	in the kernel of the restriction map
	\[
	\Hom ( \ZZ_S^\mathfrak I , \PPic_{\CCC / \kb} ( S ) )
	\to 
	\Hom ( (N_X)_S  , \PPic_{\CCC / \kb} ( S ) )  .
	\]
	This last fact gives us a trivialisation and finally an isomorphism class of collection.
	One can check that this class does not depend on the choices made. 
	\end{proof}	

\begin{mycor}
	In $\KVar{ (\DDiv_{\CCC / \kb } )^\mathfrak I}$ and $\KVar{ ( \PPic_{\CCC / \kb } )^\mathfrak I}$
	we have the equality of classes 
	\[
	\left [
	\HHom_\kb ( \CCC , X )_U \to (\DDiv_{\CCC / \kb } )^\mathfrak I_X
	\right ]
	=
	\left [
	\widetilde X \to (\DDiv_{\CCC / \kb } )^\mathfrak I_X
	\right ] .
	\]
\end{mycor}

%%%%%%%%%%%%%%%%%%%%%%%%%%%%%%%%%%%%%%%%%%%%%%%%
%%%%%%%%%%%%%%%%%%% SECTION %%%%%%%%%%%%%%%%%%%%
%%%%%%%%%%%%%%%%%%%%%%%%%%%%%%%%%%%%%%%%%%%%%%%%

\section{Curves on smooth split projective toric varieties (the classical case)}
\label{section-toric-classic}

In this section we assume that $X = X_\Sigma$ 
is a smooth projective split toric variety
defined by a complete regular fan $\Sigma$.
In that situation, 
the definition of $\Sym^\dd_\kb ( \CCC )_X$
only depends on the fan $\Sigma$,
and there are no algebraic relations between the generators of the Cox ring of $X_\Sigma$, 
so it is natural to replace the index $X$ by $\Sigma$ everywhere in our notations.
The motivic height zeta function
\[
\mathcal Z_{X_\Sigma} ( \TT )
=
\sum_{\dd \in \NN^\mathfrak I}
\left [
\HHom^\dd_\kb ( \CCC , X_\Sigma ) 
\longrightarrow 
\Sym_{ / \kb}^\dd ( \CCC ) 
\right ]_\kb
\TT^\dd 
\]
comes from the motivic function
\[
\left [
\HHom_\kb ( \CCC , X_\Sigma )
\longrightarrow (\DDiv_{\CCC / \kb})^\mathfrak I
\right ] 
\in \KVar{(\DDiv_{\CCC / \kb})^\mathfrak I}.
\]
The open subset 
\[
\Sym^\dd_\kb ( \CCC )_{\Sigma,**}
\]
is given by the condition 
\[
\forall \, J \in B_\Sigma ,
\quad
\bigcap_{j\in J} \Supp ( D_j ) = \varnothing 
\]
where 
\[
B_\Sigma 
=
\{ 
J \subset \Sigma ( 1 ) 
\mid 
\forall \, \sigma \in \Sigma ,
\; 
J \cap ( \Sigma ( 1 ) - \sigma ( 1 ) ) \neq \varnothing 
\}  
\]
is the 
set of $J\subset \Sigma ( 1 ) = \mathfrak I$
such that 
\[
\cap_{\alpha \in J} D_\alpha = \varnothing .
\]

Then for all $\dd \in \NN^\mathfrak I$
\[
\Sym^\dd_\kb ( \CCC )_\Sigma
=
\Sym^\dd_\kb ( \CCC )_{\Sigma,*}
\cap
\Sym^\dd_\kb ( \CCC )_{\Sigma,**}
\]
viewed above $\Sym^\dd_\kb ( \CCC )$ via the inclusion map.

\subsection{Möbius functions in the classical case}
The following does not depend on the fact that we take $\mathfrak I = \Sigma ( 1 ) $ in our application.

The local Möbius function 
\[
\mu_{B_\Sigma} : \NN^\mathfrak I
\to \ZZ 
\]
is defined by the relation 
\[
\mathbf 1_{A ( B_\Sigma ) } ( \nn )
=
\sum_{0 \leqslant \nn ' \leqslant \nn} \mu_{B_\Sigma} ( \nn ' ) 
\]
where $A ( B_\Sigma )$ is the set of tuples not lying above $B_\Sigma$.
The global Möbius function 
is the motivic function on $\DDiv^\mathfrak I_{ \CCC / \kb}$
defined by the coefficients of the motivic Euler product 
\[
\prod_{p\in \CCC}
\left ( 
\sum_{\ee \in \NN^\mathfrak I} \mu_{B_\Sigma} ( \ee ) \TT^\ee 
\right ) .
\]
More precisely, for every $\dd \in \NN^\mathfrak I$
we see the degree $\dd$ coefficient of this motivic Euler product 
as a class $\mu_\Sigma ( \dd ) \in \KVar {\Sym^\dd_{/\kb} ( \CCC )}$
and we see the entire family of coefficients as a class
\[
\mu_\Sigma \in \KVar{(\DDiv_{\CCC / \kb} )^\mathfrak I} 
\]
that is to say a \emph{motivic function on $( \DDiv_{\CCC / \kb} )^\mathfrak I $} defined by the relation 
\[
\left [
\begin{tikzcd}
	\mu_\Sigma
	\boxtimes_{( \DDiv_{\CCC / \kb} )^\mathfrak I}
	( \DDiv_{\CCC / \kb} )^\mathfrak I \dar \\
	( \DDiv_{\CCC / \kb} )^\mathfrak I \times_\kb ( \DDiv_{\CCC / \kb} )^\mathfrak I \dar["+"]\\
	( \DDiv_{\CCC / \kb} )^\mathfrak I
\end{tikzcd}
\right ]
=
\left [
\begin{tikzcd}
( \DDiv_{\CCC / \kb} )^\mathfrak I_{X,**} \dar \\
	( \DDiv_{\CCC / \kb} )^\mathfrak I
	\end{tikzcd}
\right ]
.\]

\begin{mylemma}
    The series
    \[
    \sum_{\dd \in \NN^{\mathfrak I }} \mu_\Sigma ( \dd ) \LL_\kb^{- | \dd |}
    \]
    converges 
    in $\widehat{\mathscr M_\kb}^{\dim}$
    to the motivic Euler product 
    \[
    \prod_{p\in \CCC}
    \left ( 
    \left (
        1 - \LL_{\kappa p } 
    \right )^{\rk ( \Pic ( X ) )}
    \frac{[X_{\kappa ( p )}]}{\LL_{\kappa ( p )}^{\dim ( X )}}
    \right ) .
    \]
\end{mylemma}

\begin{proof}
	See \cite[Remark 4.5.7]{bilu-das-howe2022hadamard} or Proposition 5.1 and Remark 5.2 in \cite{faisant2023motivic-distribution}.
	In both references $\CCC = \PP^1_\kb$, but actually the argument remains valid if $\CCC$ is arbitrary:
	it is sufficient to remark that the valuation of  
	\[
	\sum_{\mm \in \NN^\mathfrak I \setminus \{ \bm 0 \}}
	\mu_{B_\Sigma} ( \bm m ) \mathrm T^{|\bm m|}
	\in \ZZ [ \mathrm T ]
	\] 
	is at least equal to $2$ \cite[Lemme 3.8]{bourqui2009produit} so that the associated motivic Euler product converges at $T=\LL_\kb^{-1}$ 
	in $\widehat{\MMM_\kb}^{\dim}$.  
\end{proof}

\subsection{An example: morphisms from $\CCC $ to $\PP^1_\kb$}
In that case, $\Pic ( \PP^1_\kb ) = \ZZ$ and ${\Sigma ( 1 ) = \{ 0 , \infty  \}}$. The condition on linear classes 
means that 
we sum over points lying above the diagonal of the morphism
\[
\begin{array}{rll}
 \DDiv_{\CCC / \kb} \times_\kb \DDiv_{\CCC / \kb} 
& \longrightarrow
&\PPic_{\CCC / \kb} \times_\kb \PPic_{\CCC / \kb} \\
 ( D_0 , D_\infty ) & \longmapsto  & ( [\OOO_\CCC ( D_0 )] ,  [\OOO_\CCC ( D_\infty )] )  .
\end{array}
\]
The 
fibre above $( [L] , [L] ) \in \PPic_{\CCC / \kb}^d \times_\kb \PPic_{\CCC / \kb}^d$, for a certain $d \in \NN^*$,
of our parameter space 
is the subspace of tuples of sections
\[
( H^0 ( \CCC , L  )^*
\times 
H^0 ( \CCC , L )^* )_*
\]
where the lower $*$ stands for the coprimality conditions.
After a Möbius inversion, the motivic sum one has to evaluate will be of the form 
\[
\sum_{[L]\in \PPic_{\CCC / \kb}} 
\sum_{ (  E_0 , E_1 ) \in \DDiv_{\CCC / \kb }} 
\mu_{\PP^1_\kb} ( E_0 , E_1 )
\left [H^0 \left ( \CCC , L \otimes \OOO_\CCC ( -E_0 ) \right )^*\right ]
\left [
H^0 \left ( \CCC , L \otimes \OOO_\CCC ( - E_1 ) \right )^* \right ]
T^{\deg ( L ) } .
\]
where $[ L_0 ]$ and $[ L_1 ]$ are linear classes of two effective divisors $E_0,E_1$
to which the motivic Möbius function is applied.

Concretely, the computation goes as follows. 
\begin{align*}
& ( \LL_\kb - 1 ) \mathcal Z_{\PP^1_\kb} ( T ) \\
& = 
\sum_{
[L]\in \Pic ( \CCC )
}
\left [ ( H^0 ( \CCC , L  )^*
\times 
H^0 ( \CCC , L )^* )_* \right ] T^{\deg ( L ) } \\
& 
= 
\sum_{
\substack{
(E_0,E_1) \in \Div_+ ( \CCC )^2 \\ [L] \in \Pic ( \CCC ) 
}
}
\mu ( E_0 , E_1 ) 
[ H^0 ( \CCC , L \otimes \mathscr O ( - E_0 ) )^*
\times 
H^0 ( \CCC , L \otimes \mathscr O ( - E_1 ) )^*  ]
T^{\deg ( L )}
\end{align*}
Now we can replace $ [ H^0 ( \CCC , L \otimes \mathscr O ( - E_i ) )^* ]$
by 
\[
\LL_\kb^{h^0 ( \CCC , L \otimes \mathscr O ( - E_i ) )} - 1
\]
for $i=0,1$ and any $[L] \in \Pic ( \CCC )$ such that $\deg ( L ) - \deg ( E_i ) \geqslant 0$
(remember 
that we stratified $\Pic^d ( \CCC ) $
to get locally free sheaves as well for small degrees, so that our computations make sense).
If the later condition is not satisfied, the class is actually zero. 
By Riemann-Roch, we know that 
\[
h^0 ( \CCC , L \otimes \mathscr O ( - E_i ) )
= \deg ( L ) - \deg ( E_i ) + 1 - g 
\]
whenever $ \deg ( L ) - \deg ( E_i )  > 2g - 2$.

In the end the main normalised term of $\mathcal Z_{\PP^1_\kb} ( T ) $,
that is to say its coefficient of degree $d$ 
normalized by  
$\LL_\kb^{2d}$,
equals 
\[
\LL_\kb^{1-g}
\frac{
[ \Pic^0 ( \CCC ) ] \LL_\kb^{1-g}}
{\LL_\kb - 1 } 
\sum_{\bm{E} \in \left (\DDiv_{\CCC / \kb}\right )^2}
\mu ( \bm E ) \LL_\kb^{-\deg ( \bm E )}
\]
up to a negligeable error term, as $d \to \infty$.
This is exactly what we want, since 
\[
\sum_{\bm{E} \in \left (\DDiv_{\CCC / \kb}\right )^2}
\mu ( \bm{E} ) \LL_\kb^{-\deg ( \bm E )}
=
\prod_{p\in \CCC} \left ( 1 - \LL_\kp^{-2} \right )
=
\prod_{p\in \CCC} \left  ( 1 - \LL_\kp^{-1} \right ) 
\frac{[\PP^1_{\kappa ( p )}]}{\LL_{\kappa ( p )}}
\]
is the motivic Tamagawa number of $\PP^1_{\kb ( \CCC )}$ 
with respect to its isotrivial model $\PP^1_\kb \times_\kb \CCC$.

\subsection{The classical case}
For the convenience of the reader, 
we present the first steps 
of the decomposition of the motivic Height zeta function without taking Campana-type conditions into account. 
Recall that $\widehat M$, $\widetilde M$, $\widehat X$ and $\widetilde X$ were defined \cpageref{definition-X-hat}
and that they allow us to parametrise schematic points of $\HHom_\kb ( \CCC , X_\Sigma )$
and their lift to the universal torsor $\mathcal T_\mathcal L$
thanks to \cref{thm-parametrisation-finale},
at least 
piecewisely with respect to $(\PPic_{\CCC / \kb})^\mathfrak I$.

\begin{myptn}Let $\underline L =(L_i)_{i\in \mathfrak I} \in ( \PPic_{\CCC / \kb} )^\mathfrak I$.
	Then in $\KVar{\kappa ( \underline L )}$ 
	\[
\left [ \widehat X_{\underline L} \to  \prod_{i\in \mathfrak I} \LinSys_{L_i / \CCC} \right ]_{\kappa ( \underline L )}
= 
\sum_{
\substack{
\bm E \in ( \DDiv_{\CCC / \kb} )^\mathfrak I\\
\bm F \in ( \DDiv_{\CCC / \kb} )^\mathfrak I\\
\bm E + \bm F \in \prod_{i\in \mathfrak I} \LinSys_{L_i / \CCC}
}
}
\mu_X ( \bm E ) 
\left 
[
\prod_{i\in \mathfrak I }
H^0 ( \CCC , \OOO_\CCC ( F_i ) ) \setminus \{ \bm 0 \}
\right ] .
	\]
\end{myptn}
\begin{proof}
	The relation is trivial 
	if one of the $L_i$'s is not effective, 
	so we can assume that it is not the case. 
	Then for any $\bm F \in ( \DDiv_{\CCC / \kb} )^\mathfrak I$,
	one sees 
	\[
	\prod_{i\in \mathfrak I }
H^0 ( \CCC , \OOO_\CCC ( F_i )) 
	\]
	as the fibre of 
	\[
	\widehat M =
\prod_{i\in \mathfrak I} \OOO_{\PP ( \mathscr Q )} ( -1 ) ^\times 
\to \prod_{i\in \mathfrak I} \PP ( \mathscr Q ) \simeq ( \DDiv_{\CCC / \kb } )^\mathfrak I \to ( \PPic_{\CCC / \kb } )^\mathfrak I
	\]
	above  $ ( [ \OOO_\CCC (  F_i ) ])_{i\in \mathfrak I}$
	so that we understand the right member of the equality 
	as the motivic sum of 
	the restriction of the motivic function
	\[
	\begin{tikzcd}
		\mu_X \boxtimes_{( \DDiv_{\CCC / \kb} )^\mathfrak I} \widehat M \dar \\
( \DDiv_{\CCC / \kb} )^\mathfrak I 
\times_\kb ( \DDiv_{\CCC / \kb} )^\mathfrak I  \dar["+"]  \\
( \DDiv_{\CCC / \kb} )^\mathfrak I 	
\end{tikzcd}
	\]
to the constructible subset 
$\prod_{i\in \mathfrak I} \LinSys_{L_i} \hookrightarrow ( \DDiv_{\CCC / \kb} )^\mathfrak I $.
Then the identity follows from the definition of the Möbius function $\mu_X$ telling us exactly that
\[
\left [
\begin{tikzcd}
	\mu_X 
	\boxtimes_{( \DDiv_{\CCC / \kb} )^\mathfrak I}
	( \DDiv_{\CCC / \kb} )^\mathfrak I \dar \\
	( \DDiv_{\CCC / \kb} )^\mathfrak I \times_\kb ( \DDiv_{\CCC / \kb} )^\mathfrak I \dar["+"]\\
	( \DDiv_{\CCC / \kb} )^\mathfrak I
\end{tikzcd}
\right ]
=
\left [
\begin{tikzcd}
( \DDiv_{\CCC / \kb} )^\mathfrak I_{X,**} \dar \\
	( \DDiv_{\CCC / \kb} )^\mathfrak I
	\end{tikzcd}
\right ]
\]
together with the following facts: $\widehat X$ is by definition the pull-back of $\widehat M$ 
to $( \DDiv_{\CCC / \kb} )^\mathfrak I_{X,**}$
and $\widehat M$ is constant as a motivic function on $( \DDiv_{\CCC / \kb} )^\mathfrak I$. 
\end{proof}

Summing over $(\PPic_{\CCC / \kb})^\mathfrak I_X$, 
a coarse consequence of this identity is the expression
\begin{equation}
\frac{\mathcal Z_{X_\Sigma} ( \TT )}{( \LL_\kb - 1 )^{\dim ( X_\Sigma )} }
= 
\sum_{
\substack{
\dd \in \CEff ( X_\Sigma )^\vee \cap \Pic ( X_\Sigma )^\vee  \\
( d ' , d'' ) \in \NN^\mathfrak I \times \NN^\mathfrak I\\
d = d' + d'' }
}
\left [
\Sym^{d'}_\kb ( \CCC ) \boxtimes 
\mu_\Sigma  ( d'' ) 
\to \Sym^\dd_\kb ( \CCC )_{\Sigma,*}
\right ]
\TT^d 
\end{equation}
for the motivic zeta function of $X_\Sigma$.
But for now we keep working over a given class $\underline L \in (\PPic_{\CCC / \kb})^\mathfrak I$ in order to obtain a finer result. 

\bigskip

Let us replace for a while 
\[
\left [
\prod_{i\in \mathfrak I }
H^0 ( \CCC , \OOO_\CCC ( F_i ) ) \setminus \{ \bm 0 \}
\right ]
\]
by 
\[
\prod_{i \in \mathfrak I}
\left ( 
\LL_{\kappa ( \underline L )}^{
\deg ( L_i ) - \deg ( E_i ) - g  + 1 
} - 1 
\right )
\]
whenever it makes sense. 
Then the main term will be of magnitude 
\begin{align*}
& \LL_{\kappa ( \underline L )}^{
\sum_{
i \in \mathfrak I 
}
\deg ( L_i ) - \deg ( E_i ) - g  + 1 } \\
& =
\LL_{\kappa ( \underline L )}^{  - \deg ( \bm E )  }
\LL_{\kappa ( \underline L )}^{ \deg ( \underline L ) } 
\LL_{\kappa ( \underline L )}^{| \mathfrak I | (1-g)}
\end{align*}
and 
the main 
term will be 
\[
\sum_{
\substack{
\bm E \in ( \DDiv_{\CCC / \kb} )^\mathfrak I_{\kappa ( \underline L )}\\
\bm 0 \leqslant \degg ( \bm E ) \leqslant \degg ( \underline L )
}
}
\mu_\Sigma ( \bm  E )
\LL_{\kappa ( \underline L )}^{  - \deg ( \bm E )  }
\LL_{\kappa ( \underline L )}^{ \deg ( \underline L ) } 
\LL_{\kappa ( \underline L )}^{| \mathfrak I | (1-g)}
\]
which after normalisation by 
$\LL_{\kappa ( \underline L )}^{ \deg ( \underline L ) } 
\LL_{\kappa ( \underline L )}^{| \mathfrak I | (1-g)}$
and as 
\[
\degg ( \underline L ) \to \infty 
\]
will become 
\[
\sum_{\bm E \in ( \DDiv_{\CCC / \kb} )^\mathfrak I_{\kappa ( \underline L )}}
\mu_\Sigma ( \bm E )
\LL_{\kappa ( \underline L )}^{ - \deg ( \bm E ) }
\]
which is known to be convergent: it is equal to
the base-change to ${\kappa ( \underline L )}$ of 
the motivic Euler product 
\[
\prod_{v\in \CCC } 
\left (
1 - \LL_v^{-1}
\right )^r \frac{[ X_\Sigma ]}{\LL_v^n} .
\]
In general, it is natural to write 
\[
\left [ \widehat X_{\underline L} \right ]	
= 
\sum_{A \subset \mathfrak I}
( - 1 )^{| \mathfrak I \setminus A |}
\sum_{
\substack{
\bm E \in ( \DDiv_{\CCC / \kb} )^\mathfrak I_{\kappa ( \underline L )}\\
\bm 0 \leqslant \degg ( \bm E ) \leqslant \degg ( \underline L )
}
}
\mu_\Sigma ( \bm  E )
\prod_{i\in A}
\LL_{\kappa ( \underline L )}
^{h^0 ( L_i \otimes \OOO_\CCC ( - E_i ) )}
\]
and the main contribution comes from $A = \mathfrak I$. 
Now let us prove this claim rigorously. Naively, 
\[
\LL_{\kappa ( \underline L )}
^{h^0 ( L_i \otimes \OOO_\CCC ( - E_i ) )}
=
\LL_{\kappa ( \underline L )}
^{  \deg ( L_i ) - \deg ( E_i ) + 1 - g}
+ 
\underbrace{
\LL_{\kappa ( \underline L )}
^{h^0 ( L_i \otimes \OOO_\CCC ( - E_i ) )}
- 
\LL_{\kappa ( \underline L )}
^{  \deg ( L_i ) - \deg ( E_i ) + 1 - g }
}_{\substack{=0 \\
\text{ whenever $ \deg ( E_i ) < \deg ( L_i ) + 2(1 - g )$}}}
\]
so that we can write 
\begin{align}
& \sum_{
\substack{
\bm E \in ( \DDiv_{\CCC / \kb} )^\mathfrak I_{\kappa ( \underline L )}\\
\bm 0 \leqslant \degg ( \bm E ) \leqslant \degg ( \underline L )
}
}
\mu_\Sigma ( \bm  E )
\prod_{i\in A}
\LL_{\kappa ( \underline L )}
^{h^0 ( L_i \otimes \OOO_\CCC ( - E_i ) )}\notag\\
& = 
\sum_{
\substack{
\bm E \in ( \DDiv_{\CCC / \kb} )^\mathfrak I_{\kappa ( \underline L )}\\
\bm 0 \leqslant \degg ( \bm E ) \leqslant \degg ( \underline L )
}
}
\mu_\Sigma ( \bm  E )
\LL_{\kappa ( \underline L )}
^{\sum_{i\in A}  \deg ( L_i ) - \deg ( E_i ) + 1 - g }
\label{equation-first-term}
\\
& + 
\sum_{
\substack{
\bm E \in ( \DDiv_{\CCC / \kb} )^\mathfrak I_{\kappa ( \underline L )}\\
\bm 0 \leqslant \degg ( \bm E ) \leqslant \degg ( \underline L )
}
}
\mu_\Sigma ( \bm  E )
\left ( 
\prod_{i\in A}
\LL_{\kappa ( \underline L )}
^{h^0 ( L_i \otimes \OOO_\CCC ( - E_i ) )}
- 
\prod_{i\in A}
\LL_{\kappa ( \underline L )}
^{  \deg ( L_i ) - \deg ( E_i ) + 1 - g }
\right )
\label{equation-second-term}
\end{align}
If one divides the first term \eqref{equation-first-term} by $\LL_{\kappa ( \underline L )}^{| \degg ( \underline L ) |}$,
one gets 
\[
\LL_{\kappa ( \underline L )}^{ |A| ( 1 - g ) }
\sum_{
\substack{
\bm E \in ( \DDiv_{\CCC / \kb} )^\mathfrak I_{\kappa ( \underline L )}\\
\bm 0 \leqslant \degg ( \bm E ) \leqslant \degg ( \underline L )
}
}
\mu_\Sigma ( \bm  E )
\LL_{\kappa ( \underline L )}
^{ |\deg ( \underline L )| - \degg ( \bm E )}
\LL_{\kappa ( \underline L )}
^{\sum_{i\notin A} \deg ( E_i ) - \deg ( L_i )}
\]
which is known to converge when $A = \mathfrak I $,
with error term bounded by 
\[
- \frac{1}{2} \min_{i\in \mathfrak I} ( \deg ( L_i ) + 1 )
\]
and has dimension bounded by 
\[
- \frac{1}{4} \min_{i \in A} \deg ( L_i ) +  |A| ( 1 - g ) 
\]
if $A \neq \mathfrak I $.

The second term \eqref{equation-second-term} can be decomposed again with respect to the number of components of $\bm E$ having small degree:
\[
\sum_{\varnothing \neq A ' \subset A}
\sum_{
\substack{
\bm E \in ( \DDiv_{\CCC / \kb} )^\mathfrak I_{\kappa ( \underline L )}\\
\bm 0 \leqslant \degg ( \bm E ) \leqslant \degg ( \underline L )\\
\deg ( L_i ) + 2 ( 1 - g )  \leqslant \deg ( E_i )  \\ \Leftrightarrow i \in A'
}
}
\mu_\Sigma ( \bm  E )
\left ( 
\prod_{i\in A}
\LL_{\kappa ( \underline L )}
^{h^0 ( L_i \otimes \OOO_\CCC ( - E_i ) )}
- 
\prod_{i\in A}
\LL_{\kappa ( \underline L )}
^{  \deg ( L_i ) - \deg ( E_i ) + 1 - g }
\right )
\]
where it is now possible to rewrite the term of the motivic sum corresponding to  ${\varnothing \neq A ' \subset A}$ as 
\[
\mu_\Sigma ( \bm E )
\LL_{\kappa ( \underline L )}^{\sum_{i\in A \setminus A'} \deg ( L_i ) - \deg ( E_i ) + 1 - g}
\left ( 
\prod_{i\in A '}
\LL_{\kappa ( \underline L )}
^{h^0 ( L_i \otimes \OOO_\CCC ( - E_i ) )}
- 
\prod_{i\in A '}
\LL_{\kappa ( \underline L )}
^{  \deg ( L_i ) - \deg ( E_i ) + 1 - g }
\right ) 
\]
and which after normalisation by $\LL_{\kappa ( \underline L )}^{| \degg ( \underline L ) |}$
becomes 
\begin{align*}
& \LL_{\kappa ( \underline L )}^{| A \setminus A ' | ( 1 - g )}
\mu_\Sigma ( \bm E ) 
\left ( 
\prod_{i\in A '}
\LL_{\kappa ( \underline L )}
^{h^0 ( L_i \otimes \OOO_\CCC ( - E_i ) )}
- 
\prod_{i\in A '}
\LL_{\kappa ( \underline L )}
^{  \deg ( L_i ) - \deg ( E_i ) + 1 - g }
\right ) 
\LL_{\kappa ( \underline L )}^{- | \degg ( \bm E )|}\\
& \times 
\LL_{\kappa ( \underline L )}^{\sum_{i\in A'} \deg ( E_i ) - \deg ( L_i )}
\LL_{\kappa ( \underline L )}^{\sum_{i\notin A} \deg ( E_i ) - \deg ( L_i )}\\
& =
\LL_{\kappa ( \underline L )}^{| A  | ( 1 - g )}\mu_\Sigma ( \bm E ) 
\underbrace{\left ( 
\prod_{i\in A '}
\LL_{\kappa ( \underline L )}
^{h^0 ( L_i \otimes \OOO_\CCC ( - E_i ) ) - \deg ( L_i ) + \deg ( E_i ) + g - 1}
- 
1
\right )}_{ \text{dimension bounded by } | A ' | ( g - 1 )}
\LL_{\kappa ( \underline L )}^{- | \degg ( \bm E )|}
\LL_{\kappa ( \underline L )}^{\sum_{i\notin A} \deg ( E_i ) - \deg ( L_i )}
\end{align*}
which has dimension bounded by 
\[
- \frac{1}{4} \min_{i \in A} \deg ( L_i ) +  |A \setminus A '| ( 1 - g ) 
\]
by \cite[Lemma 2.21]{faisant2023motivic-distribution}. 

\bigskip 

Putting everything together, one gets the following stronger version of the Batyrev-Manin-Peyre conjecture for morphisms $\CCC \to X_\Sigma$.
\begin{mythm}
\label{thm-BMP-C-toric}
	As 
	$\degg ( \underline L ) \to \infty $,
	the value at $\underline L \in \HHom ( \PPic_{X_\Sigma /\kb} , \PPic_{\CCC / \kb} )$ 
	of the motivic function 
	given by the class 
	\[
	\left [ \HHom_\kb^{\degg ( \scdot ) } ( \CCC , X_\Sigma )_U \right ]
	\LL^{- \degg ( \scdot ) \scdot \omega_{X_\Sigma}^{-1}  - n ( 1 - g )}
	\in \MMM_{\HHom ( \PPic_{X_\Sigma /\kb} , \PPic_{\CCC / \kb} )} 
	\]
	tends to (the 
	value at $\underline L $
	of) the
	constant motivic function
	\[
	\left ( 
	\frac{
	\LL^{( 1 - g )}}
	{\LL - 1 }
	\right )^{\rg ( \Pic ( X_\Sigma )) }	 \prod_{p\in \CCC} 
	\left ( 
	1 - \LL 
	\right )^{\rg ( \Pic ( X_\Sigma )) }	
	\frac{[X_\Sigma ]}{\LL^n}
	\in 
	\widehat{\MMM}^{\dim}
	_{\HHom ( \PPic_{X_\Sigma /\kb} , \PPic_{\CCC / \kb} ) }
	\]
	uniformly in $\underline L$,
	meaning that there exist an absolute constant $c$
	such that the error term has virtual dimension bounded by
	\[
	- \frac{1}{4} \min_{i\in \mathfrak I} \deg ( L_i ) + c. 
	\]
\end{mythm}
Summing over $\HHom ( \PPic_{X_\Sigma /\kb} , \PPic_{\CCC / \kb} )$, one gets back the usual expected motivic stabilisation over $\Spec ( \kb )$.
\begin{mycor}
For smooth split projective toric varieties,
the answer to \cref{question-motivic-BMP-isotrivial}
is positive:
as $\dd \to \infty$,
the normalised class
\[
	\left [ \HHom_\kb^\dd ( \CCC , X_\Sigma )_U \right ]
	\LL^{- \dd \scdot \omega_{X_\Sigma}^{-1}  - n ( 1 - g )}_\kb
	\in \MMM_{\kb} 
\]
tends to 
\[
\left ( 
	\frac{
	[ \PPic^0_{\CCC / \kb } ]
	\LL_\kb^{( 1 - g )}}
	{\LL_\kb - 1 }
	\right )^{\rg ( \Pic ( X_\Sigma )) }	 \prod_{p\in \CCC} 
	\left ( 
	1 - \LL_p 
	\right )^{\rg ( \Pic ( X_\Sigma )) }	
	\frac{[X_\Sigma ]}{\LL^n_p}
	\in 
	\widehat{
	\MMM_\kb}^{\dim} .
\]
\end{mycor}

\section{Campana curves on smooth split projective toric varieties}
\label{section-Campana-curves-on-split-toric-varieties}

\subsection{Generalities about Campana curves}

We start with recalling a few definitions, making precise the drawing from the introduction.  

\smallskip 

We call \emph{Campana orbifold}
a pair 
$(V , D )$
where $V$ is a smooth projective 
variety over $F = \kb ( \CCC ) $
and $D$ is an effective Weil 
$\QQ$-divisor 
on $V$
with strict normal crossings 
of the form
\[
D = \sum_{\alpha \in \AAA} \epsilon_\alpha D_\alpha
\]
where the $D_\alpha$'s 
are irreducible divisors on $V$ 
indexed by a finite set $\AAA$
and the rational numbers $\epsilon_\alpha \in [0,1]\cap \QQ$
have shape 
\[
\epsilon_\alpha 
=
1 - \frac{1}{m_\alpha} 
\qquad
m_\alpha \in \NN_{>0} \cup \{ \infty \}
\]
(with convention $\frac{1}{\infty} = 0$).
We assume furthermore that 
$K_V + D$ 
is $\QQ$-Cartier.

A \emph{good model}
$\VVV \to \CCC$ of $V$ being given (\emph{good} generally meaning flat and proper), 
we denote by $\DDD_\alpha$
the closure of $D_\alpha$
in $\VVV$
for all $\alpha \in \AAA$.

Given a closed point $p\in \CCC$,
a section $\sigma : \CCC \to \VVV$
induces an arc 
$\sigma_p : \Spec ( \widehat{\OOO_{\CCC, p}} )
\to \VVV$,
where $\widehat{\OOO_{\CCC, p}}$
is the completion of the local ring of $\CCC$ at $p$.

\begin{mydefintro}
For any section $\sigma : \CCC \to \VVV$
intersecting the complement of the $\DDD_\alpha$'s
and any closed point $p$,
the local intersection degree of $\CCC$ with $\DDD_\alpha$
at $p$
is the non-negative integer
\[
( \sigma , \DDD_\alpha )_p 
\]
given by the multiplicity
of $\sigma_p ^* \DDD_\alpha$. 
\end{mydefintro}

\begin{mydef}
Assume that a finite set $\mathsf S$ of closed points of $\CCC$
is given.
A section $\sigma : \CCC \to \VVV $
is called a \emph{Campana section}
or \emph{curve}
if for every closed point $p$ of $\CCC$ 
outside the finite set $\mathsf S$,
\begin{equation}
\label{def:campana-conditions}
	( 
\sigma , \DDD_\alpha )_p 
\in 
\{ 0 \} \cup \NN_{\geqslant m_\alpha} 
\end{equation}
for every $\alpha \in \AAA$.
\end{mydef}

\subsection{Moduli space of Campana curves}

For the sake of simplicity, from now on we assume that $\mathsf S = \varnothing $ 
and that $m_\alpha < \infty$ for all $\alpha \in \AAA$. 
Recall that sections of the isotrivial model $X \times_\kb \CCC \overset{\pr_2}{\to} \CCC$ correspond to morphisms $\CCC \to X$.

\begin{mylemma}
	The Campana conditions \eqref{def:campana-conditions}
	define a constructible subset of 
	$
	\HHom^\dd  ( \CCC , X ).
	$
\end{mylemma}

\begin{proof}
	It is shown 
	in \cite{bilu2023MAMS}
	that 
	\[
	\HHom^\dd_\kb  ( \CCC , X ) \rightarrow \Sym_\kb^\dd ( \CCC )  
	\]
	defines a map of constructible sets. 
	The image is in fact contained in the Hilbert scheme of 
	zero-dimensional subschemes of $\CCC$
    having length $\dd$. The condition on the multiplicity 
    is encoded by the varieties defining the coefficients of the motivic Euler product 
    \[
    \prod_{
p \in \CCC } 
\left (  
 1 +
 \sum_{\ee \geqslant \mm} \TT^\ee 
\right )
\] 
and our moduli space of Campana curves of degree $\dd$ 
is obtained by pulling back to the coefficient of degree $\dd$. 
\end{proof}

\subsection{Möbius functions for Campana curves}

In the Campana case,
the local valuation monoïd 
$\NN^\mathfrak I$
is replaced by 
\[
\NN^\mathfrak I_{\geqslant \mm} \cup \{ \bm{0} \} 
\]
and $A ( B_\Sigma ) $ is replaced by 
\[
A ( B_\Sigma )_\mm
= 
A ( B_\Sigma ) \cap
\left ( 
\NN^\mathfrak I_{\geqslant \mm}  
\cup \{ \bm{0} \} \right ) . 
\]
The new local Möbius function 
\[
\mu_{B_\Sigma,\mm} :
\NN^{\mathfrak I } \to \ZZ 
\]
is given by the relation 
\[
\mathbf 1_{A ( B_\Sigma )_\mm } ( \nn )
=
\sum_{0 \leqslant \nn ' \leqslant \nn} \mu_{B_\Sigma,\mm} ( \nn ' ) 
\bm 1_{ \NN^\mathfrak I_{\geqslant \mm}  
\cup \{ \bm{0} \}} ( \nn - \nn ' ) 
\]
and the global one is given by the coefficients of the motivic Euler product
\[
\prod_{
p \in \CCC } 
\left ( 
P_{B_\Sigma}^\epsilon ( \TT )
\right ) 
\]
where 
\[
P_{B_\Sigma}^\epsilon ( \TT )
= \sum_{\mm \in \NN^\mathfrak J} \mu_{B_\Sigma,\mm} ( \mm ) \TT^\mm. 
\]

\begin{myptn}
\label{proposition-convergence-of-motivic-mobius-function-campana}
There exists an $\eta > 0$ 
such that the formal series
	\[
\prod_{
p \in \CCC 
} 
\left ( 
\prod_{ 
i \in \mathfrak I   
}
( 1 - t_i^{m_i} )
\sum_{\nn \in A ( B_\Sigma )_\mm} \TT^\nn 
\right ) 
\]
converges for $|t_i|  < \LL_\kb^{-1/m_i + \eta} = \LL_\kb^{-( 1 - \epsilon_i ) + \eta }$,
with value 
at $t_i = \LL_\kb^{-1/m_i }$
equal to the motivic Euler product with local factor equal to
\[
(1-\LL_p)^{\rk ( \Pic ( X_\Sigma ))} 
	\int_{
	\substack{
	\LLL_\infty ( X_\Sigma ) \\ \ord_{\mathcal D_\alpha} \in \NN_{\geqslant m_\alpha} \cup \{ 0 \}
	}
	} 
	\mathrm d \mu_{( X_\Sigma , D_{\bm \epsilon} )}
	=
	1 
	+ 
	\sum_{\nn \in A ( B_\Sigma )_{\mm} \setminus \{ \bm 0 \}}
	\LL_p^{- | \frac{\nn}{\mm} | }
	\prod_{i\in \mathfrak I} ( 1 - \LL_p^{-1} ) 
\]
and the $\ee$-th 
error term 
having virtual dimension at most 
\[
- \frac{1}{4}
	\min_i 
	\left ( 
	(e_i + 1 ) ( 1 - \epsilon_i )	
	\right ) .
\]
\end{myptn}

\begin{proof}Since
\[
\prod_{ 
i \in \mathfrak I   
}
( 1 - t_i^{m_i} )
\sum_{\nn \in A ( B_\Sigma )_\mm} \TT^\nn 
=
\left ( \prod_{ 
i \in \mathfrak I   
}
( 1 - t_i^{m_i} )
\right ) 
\times 
P_{B_\Sigma}^\epsilon ( \TT ) 
\times 
\sum_{\mm \in \NN^\mathfrak I_{\geqslant \mm}  
\cup \{ \bm{0} \}} \TT^\mm 
\]
	we have to control the terms 
	coming from $P_{B_\Sigma}^\epsilon ( \TT ) $
	on one hand, 
	and the ones coming from 
	$\sum_{\nn \in \NN^\mathfrak I  }
	\bm 1_{ \NN^\mathfrak I_{\geqslant \mm}  
\cup \{ \bm{0} \}} ( \nn ) \TT^\nn $
	on the other hand. 
	
	Concerning $P_{B_\Sigma}^\epsilon ( \TT ) $, 
	it is enough to remark that 
	if $\mu_{B_\Sigma}^\varepsilon ( \ee ) \neq 0$,
	then at least two coordinates of $\ee$ are non-zero.
	Moreover, if $i$ is such a coordinate, then $e_i \geqslant m_i$. 
	Now we can apply \cref{lemma:convergence-criterion-Z-coeff-multivariable}
	as follows.
	We see $P_{B_\Sigma}^\epsilon ( \TT )$
	as a power series 
	$P ( \uu )$
	in a new set of indeterminates 
	$(u_i )_{i \in \mathfrak I}$
	given by $u_i = t_i^{m_i}$,
	that is to say
	$ t_i =  u_i ^{1/m_i}$.
	This gives that the $\ee$-th coefficient 
	($\ee \in \NN^\mathfrak I$)
	of 
	\[
	\prod_{p \in \CCC} P_{B_\Sigma}^\epsilon ( \TT )
	\]
	has dimension bounded by 
	\[
	\frac{1}{2}
	\sum_{i \in \mathfrak I} \frac{e_i}{m_i}
	= 
	\frac{1}{2}
	| ( \bm 1 - \bm  \epsilon ) \ee |
	\]
	(where the product of tuples of numbers is taken coordinate-wise).
	In particular, it converges at $t_i = \LL^{-1/m_i} = \LL^{-( 1 - \epsilon_i )}$, 
	with error term of virtual dimension bounded by
	\[
	- \frac{1}{2}
	\min_i 
	\left ( 
	\frac{e_i + 1 }{m_i}
	\right ) 
	=
	- \frac{1}{2}
	\min_i 
	\left ( 
	(e_i + 1 ) ( 1 - \epsilon_i )	
	\right ) 
	.
	\]
	
	Now one writes
	\begin{align*}
	\left ( 
\sum_{\nn \in \NN^\mathfrak I  }
\mathbf 1_{\NN^\mathfrak I_{\geqslant \mm} \cup \{ \bm{0} \}} ( \nn ) \TT^\nn 
\right )  
& = 
1 
+
\sum_{\substack{A\subset \mathfrak I \\ A \neq \varnothing }}
\sum_{\substack{\nn \in \NN^A \\ n_i \geqslant m_i }}
\TT^\nn   \\
& = 
1 
+
\sum_{i \in \mathfrak I} 
\sum_{n \geqslant m_i} 
t_i^n 
+ 
\sum_{\substack{A\subset \mathfrak I \\ | A | \geqslant 2 }}
\sum_{\substack{\nn \in \NN^A \\ n_i \geqslant m_i }}
\TT^\nn \\
& = 
1 
+
\underbrace{\sum_{i \in \mathfrak I} 
t_i^{m_i}}_{(*)}
+
\sum_{i \in \mathfrak I} 
\sum_{n > m_i} 
t_i^n 
+ 
\sum_{\substack{A\subset \mathfrak I \\ | A | \geqslant 2 }}
\sum_{\substack{\nn \in \NN^A \\ n_i \geqslant m_i }}
\TT^\nn
	\end{align*}
as well as 
\begin{align*}
	\prod_{i \in \mathfrak I} 
( 1 - t_i^{m_\alpha} )
= 
1 
- 
\underbrace{
\sum_{i \in \mathfrak I}
t_\alpha^{m_\alpha} }_{(*)}
+ 
\sum_{\substack{A\subset \mathfrak I \\ | A | \geqslant 2 }}
( - 1 )^{|A|} \prod_{\alpha \in A} t_i^{m_i} 
\end{align*}
so that it becomes clear that 
if the $\nn$-th coefficient ($\nn \neq \mathbf 0$)
of 
the product
\begin{align*}
	\left ( 
\sum_{\nn \in \NN^\mathfrak I  }
\mathbf 1_\epsilon ( \nn ) \TT^\nn 
\right ) 
\prod_{i \in \mathfrak I} 
( 1 - t_\alpha^{m_\alpha} ) 
\end{align*}
is non-zero, 
then at least two coordinates of $\nn$ are non-trivial and verifies $n_\alpha \geqslant m_\alpha$. 
Indeed, the two occurences of $(*)$ compensate when one expands the product. 
We apply \cref{lemma:convergence-criterion-Z-coeff-multivariable} once again 
in the same way we did for $P^\epsilon_{B_\Sigma} ( \TT )$.
The factor $1/4 = 1/2 \times 1/2$ comes from the fact that we work 
with a product of two convergent series,
both having coefficients of dimension bounded by 
\[
\frac{1}{2}
	| ( \bm 1 - \bm \epsilon )  \ee |
\]
see \cite[Lemma 3.4]{faisant2023geometric-BMP-Ga} (replace the weight by the dimension and take $\bm \rho =  \bm 1 - \bm  \epsilon $).

The computation of the local factor is a straightforward 
adaptation of \cite[Remark 5.2.]{faisant2023motivic-distribution}.
\end{proof}

\subsection{Restricting our decomposition to Campana curves}

For all $\dd \in \NN^\mathfrak I$,
one replaces $\Sym^\dd_\kb ( \CCC )_\Sigma
 $
by 
its subscheme
\[
\Sym^\dd_\kb ( \CCC )_{\Sigma , \mm} 
\]
consisting 
of all the $\mathfrak I$-tuples of effective divisors
on $\CCC$ of degree $\dd\in \NN^\mathfrak I$
satisfying the additional condition 
\[
D_i =
\sum_v k_{i,v} [v]
\quad k_{i,v} \in \{ 0 \} \cup \mathbf N_{\geqslant m_i}  
\quad \forall \, v 
\quad 
\forall i \in \mathfrak I . 
\]
We define $(\DDiv_{\CCC / \kb} )^\mathfrak I_{\bm m}$ similarly.

\begin{mythm}
\label{thm-Campana-final}
Let $X_\Sigma$ be a smooth split and projective toric variety over $\kb$
and $\CCC$ be a smooth projective curve over $\kb$ admitting a $\kb$-divisor of degree one. 
Then 

\label{thm-general-campana}
	\begin{align*}
	& \left [
	\HHom_\kb^\mdeg (  \CCC  , ( X_\Sigma , \mathcal D_\epsilon )  )_U
	\right ]
	\LL_{\kb}^{- \mdeg \scdot \omega_V^{-1} ( \mathcal D_\epsilon ) - n(1-g)} \\
	& \longrightarrow 
    \left ( \frac{ [\Pic^0 ( \CCC ) ] \LL^{1-g} }{ \LL_\kb - 1 } \right )^{\rk ( \Pic ( V ))} 	
    \prod_{p\in \CCC } 
	( 1 - \LL_p^{-1} )^r 
	\int_{\LLL_\infty ( X_\Sigma )} \mathbf 1_\epsilon \mathrm d \mu_{X_\Sigma}
	\\
	\end{align*}
	as
	$d ( \mdeg , \partial \CEff ( V )^\vee ) \to \infty $,
	for the dimensional filtration. 
\end{mythm}

\begin{proof}
One could work relatively to $( \PPic_{\CCC / \kb} )^\mathfrak I$ but to simplify we work directly over $\kb$ and 
apply \cref{proposition-convergence-of-motivic-mobius-function-campana}. 
Given a multidegree $\dd \in \NN^\mathfrak I$,
after a Möbius inversion,
we have to evaluate
\begin{equation}
\sum_{
\substack{
\bm E \in ( \DDiv_{\CCC / \kb} )^\mathfrak I\\
\bm F \in ( \DDiv_{\CCC / \kb} )^\mathfrak I_{\bm m}}
}
\mu_{X,\mm} ( \bm E ) 
\left 
[
\prod_{i\in \mathfrak I }
H^0 ( \CCC , \OOO_\CCC ( F_i ) ) \setminus \{ \bm 0 \}
\right ] 
\end{equation}
and then divide it by $\left [ ( \PPic_{\CCC / \kb}^0 )^n \right ]$.
Up to a power of $[\mathbf G_m]$,
it
is exactly the coefficient of degree $\dd$
of 
\[
\prod_{i\in \mathfrak I} Z_{\CCC}^\Kapr ( T_i^{m_i} )
\times 
\prod_{v \in \CCC}
\left ( 
\prod_{i\in \mathfrak I}
( 1 - T_i^{m_i} ) 
\times 
P_{B_\Sigma}^\epsilon ( \TT ) 
\times 
\sum_{\mm \in \NN^\mathfrak I_{\geqslant \mm}  
\cup \{ \bm{0} \}} \TT^\mm 
\right )
\]
which can be rewritten 
\begin{equation}
\label{equation-Campana-final-expression}
\sum_{\ee + \mm ( \ee ' + \ee '' ) = \dd }
\mathfrak c ( \ee  ) \mathfrak f ( \ee ' ) \LL^{ | \ee '' |}
= 
\sum_{\ee + \mm ( \ee ' + \ee '' ) = \dd } 
\mathfrak c ( \ee ) \mathfrak f ( \ee ' ) \LL^{ - | (\ee + \mm \ee ') /\mm | } \LL^{| \dd / \mm |}
\end{equation}
where the $\mathfrak c ( \ee  )$'s are the coefficient of the right factor, a motivic Euler product 
which is known to be convergent at $T_i = \LL^{-1/m_i}$
by \cref{proposition-convergence-of-motivic-mobius-function-campana},
and the $\mathfrak f ( \ee ' )$'s
are the coefficients of 
\[
\prod_{i\in \mathfrak I}
\frac{ P_{\CCC} ( T_i^{m_i} ) }{ 1 - T_i^{m_i} } 
\]
with $P_{\CCC} ( T ) $ being the polynomial of degree $2g$ such that 
\[
Z_\CCC^\Kapr ( T ) = \frac{P_{\CCC} ( T ) }{( 1 - T )( 1 - \LL T )} ,
\]
see for example Theorem 1.3.1 in \cite[Chap. 7]{chambert-loir-nicaise-sebag2018motivic}.
Each of the terms
involved in \eqref{equation-Campana-final-expression}
 being convergent at $T_i = \mathbf L_i^{1/m_i}$, after normalisation 
of \eqref{equation-Campana-final-expression}
by $ \LL^{| \dd / \mm |}$,
one gets the result. 
\end{proof}

%%%%%%%%%%%%%%%%%%%%%%%%%%%%%%%%%%%%%%%%%%%%%%%%
%%%%%%%%%%%%%%%%%%% SECTION %%%%%%%%%%%%%%%%%%%%
%%%%%%%%%%%%%%%%%%%%%%%%%%%%%%%%%%%%%%%%%%%%%%%%

%\section{Equidistribution of curves on toric varieties}
%\label{section-equidistribution}

%In this section we explain how to extend the main result of \cite[\S 5]{faisant2023motivic-distribution}
%to morphisms from an arbitrary curve $\CCC$ to a projective variety. 

%\begin{center}
%\includegraphics[scale=0.1]{in-progress.jpg}  
%\end{center}

% TODO [X] Écrire cette section plus tard 

%%%%%%%%%%%%%%%%%%%%%%%%%%%%%%%%%%%%%%%%%%%%%%%%
%%%%%%%%%%%%%%%%%%% SECTION %%%%%%%%%%%%%%%%%%%%
%%%%%%%%%%%%%%%%%%%%%%%%%%%%%%%%%%%%%%%%%%%%%%%%

%\section{On intrinsic linear hypersurfaces} 

%\label{section-intrinsic-hypersurfaces}

%%%%%%%%%%%%%%%%%%%%%%%%%%%%%%%%%%%%%%%%%%%%%%%%
%%%%%%%%%%%%%%% BIBLIOGRAPHY %%%%%%%%%%%%%%%%%%%
%%%%%%%%%%%%%%%%%%%%%%%%%%%%%%%%%%%%%%%%%%%%%%%%

\bibliography{references}

%%%%%%%%%%%%%%%%%%%%%%%%%%%%%%%%%%%%%%%%%%%%%%%%
%%%%%%%%%%%%%%%%%%% END DOC %%%%%%%%%%%%%%%%%%%%
%%%%%%%%%%%%%%%%%%%%%%%%%%%%%%%%%%%%%%%%%%%%%%%%

\end{document}

%% file: macros.tex
\DeclareMathOperator{\Sym}{\mathbf{Sym}}

\DeclareMathOperator{\CEff}{\mathrm{Eff}} %Effective cone (over R)
\DeclareMathOperator{\Hom}{\mathrm{Hom}} % Set of homomorphisms  / or scheme repres. functor
\DeclareMathOperator{\HHom}{\mathbf{Hom}} % Functor of homomorphisms 
\DeclareMathOperator{\Spec}{\mathrm{Spec}} % Spectrum of a ring 
\DeclareMathOperator{\Pic}{\mathrm{Pic}} % Picard group
\DeclareMathOperator{\Div}{\mathrm{Div}} %Relative EffDivFunctor
\DeclareMathOperator{\PPic}{\mathbf{Pic}} % Relative Picard functor
\DeclareMathOperator{\DDiv}{\mathbf{Div}} %Relative EffDivFunctor
\DeclareMathOperator{\NS}{\mathrm{NS}} % Néron-Severi group 
\DeclareMathOperator{\Cox}{\mathrm{Cox}} % Cox ring  
\DeclareMathOperator{\LinSys}{\mathrm{LinSys}} % Hilbert scheme

\DeclareMathOperator{\pr}{\mathrm{pr}} % Projection 
\DeclareMathOperator{\Supp}{\mathrm{Supp}} % Support 
\DeclareMathOperator{\res}{\mathrm{res}} % Residue 
\DeclareMathOperator{\rk}{\mathrm{rk}}  % Rank 
\DeclareMathOperator{\rg}{\mathrm{rk}}  % Rank (french command)
\DeclareMathOperator{\ord}{\mathrm{ord}} % Order
\newcommand{\KVar}[1]{K_0  \mathbf{Var}_{#1} } %Grothendieck ring of varieties
\newcommand{\KVarmon}[1]{K_0  \mathbf{Var}^+_{#1} } %Grothendieck ring of varieties
\newcommand{\KVarpar}[1]{K_0  ( \mathbf{Var}_{#1} ) } %Grothendieck ring of varieties with parentheses  
\DeclareMathOperator{\PL}{\mathrm{PL}} % Piecewise linear functions 

\DeclareMathOperator{\degg}{\mathbf{deg}}  

\newcommand{\dd}{{\bm{d}}}  
\newcommand{\ee}{{\bm{e}}}  
  
\newcommand{\mm}{{\bm{m}}}  
\newcommand{\nn}{{\bm{n}}}  
 
\newcommand{\uu}{{\bm{u}}}   

\newcommand{\CC}{\mathbf{C}}

\newcommand{\PP}{\mathbf{P}}
\newcommand{\RR}{\mathbf{R}}
\newcommand{\ZZ}{\mathbf{Z}}
\newcommand{\NN}{\mathbf{N}}
\newcommand{\TT}{{\mathbf{T}}}
\newcommand{\LL}{\mathbf{L}}
\newcommand{\GG}{\mathbf{G}}

\newcommand{\QQ}{\mathbf{Q}}

\newcommand{\AAA}{\mathscr{A}}

\newcommand{\CCC}{\mathscr{C}}
\newcommand{\DDD}{\mathscr{D}}

\newcommand{\FFF}{\mathscr{F}}

\newcommand{\LLL}{\mathscr{L}}
\newcommand{\MMM}{\mathscr{M}}
\newcommand{\OOO}{\mathscr{O}}
\newcommand{\PPP}{\mathscr{P}}

\newcommand{\VVV}{\mathscr{V}}

\newcommand{\XXX}{\mathscr{X}}

\newcommand{\Kapr}{\mathrm{Kapr}}

\newcommand{\mdeg}{\delta}
\newcommand{\scdot}{\,\cdot \,}
\newcommand{\kb}{k} %Base field 
\newcommand{\kp}{{\kappa ( p )}}

\newcommand{\AJ}{\mathrm{AJ}}